\documentclass{amsart}

\usepackage[usenames]{color}
\usepackage[active]{srcltx}

\usepackage[margin=2.5cm]{geometry}
\usepackage[all]{xy} \xyoption{arc}\xyoption{poly}
\usepackage{amssymb}

\newtheorem{lemma}[equation]{Lemma}
\newtheorem{prop}[equation]{Proposition}
\newtheorem{thm}[equation]{Theorem}
\newtheorem{cor}[equation]{Corollary}

\newtheorem{defn}[equation]{Definition}

\theoremstyle{definition}
\newtheorem{exmp}[equation]{Example}

\newtheorem{rmk}[equation]{Remark}

\numberwithin{equation}{section}

\newcommand{\Z}{\mathbf{Z}}

\newcommand{\R}{\mathbf{R}}
\newcommand{\C}{\mathbf{C}}

\newcommand{\U}[1]{\mathrm{U}(#1)}
\newcommand{\VU}[1]{\mathrm{VU}(#1)}

\newcommand{\BU}[1]{\mathrm{BU}(#1)}

\newcommand{\map}[3]{\operatorname{map}(#2,#3)_{#1}}

\newcommand{\Map}[5]{\operatorname{map}(#2,#3;#4,#5)_{#1}}
\newcommand{\func}[3]{\mbox{$#1 \colon #2 \to #3$}}

\newcommand{\sk}[2]{\mbox{$\operatorname{sk}_{#1}(#2)$}}

\newcommand{\fla}[2]{\mbox{$\operatorname{fla}_{#1}(#2)$}}

\newcommand{\ch}[1]{\operatorname{chr}(#1)}
\newcommand{\chs}[2]{\operatorname{chr}^{#1}(#2)}
\newcommand{\DaJ}[1]{\operatorname{DJ}(#1)}

\newcommand{\SR}[1]{\operatorname{SR}(#1)}

\newcommand{\daj}{Davis--Januszkiewicz}

\newcommand{\codim}{\operatorname{codim}}
\newcommand{\colim}{\operatorname{colim}}

\newcommand{\m}{morphism}

\title{Vertex colorings of simplicial complexes}

\author{Natalia Dobrinskaya}
\address{Department of Mathematics
Vrije Universiteit\\
Faculty of Sciences\\
De Boelelaan 1081a\\
1081 HV Amsterdam\\
The Netherlands}
\email{NE.Dobrinskaya@few.vu.nl}
\urladdr{}

\author{Jesper M.~M\o ller}
\address{Matematisk Institut\\
  Universitetsparken 5\\
  DK--2100 K\o benhavn}
\email{moller@math.ku.dk}
\urladdr{htpp://www.math.ku.dk/~moller}

\author{Dietrich Notbohm}
\address{Department of Mathematics
Vrije Universiteit\\
Faculty of Sciences\\
De Boelelaan 1081a\\
1081 HV Amsterdam\\
The Netherlands}
\email{notbohm@few.vu.nl}
\urladdr{}

\usepackage[bookmarks=true,bookmarksopen=false]{hyperref}
\hypersetup{
   pdftitle = {},
   pdfauthor = {Jesper Michael Møller},
   pdfpagemode = {UseOutlines},
   pdfstartview = {FitH},
   pdfborder = {0 0 0},
   backref = {true},
   colorlinks = {true},
   urlcolor = {blue},
   citecolor = {blue},
   linkcolor = {blue},
   pdftoolbar = {true},
}

\begin{document}
\date{\today}
\maketitle
\tableofcontents

\section{Introduction}
\label{sec:intro}

Let $K$ be a finite abstract simplicial complex (ASC), $V$ the vertex
set of $K$, and $P$ a finite set of colors. We say that $K$ is
$(P,s)$-colorable if it is possible to paint $V$ with colors from the
palette $P$ in such a way that no simplex of $K$ contains more than
$s$ vertices of the same color. In other words, a $(P,s)$-coloring of
$K$ is a map \func fVP so that $|\sigma \cap f^{-1}p| \leq s$ for all
simplices $\sigma$ of $K$.  A $(P,1)$-coloring is usually called just
a coloring.  Section~\ref{sec:ideal} and Section~\ref{sec:gencol}
contain more detailed discussions of usual colorings and our
relaxed colorings.

Figure~\ref{fig:MBcol} shows a $2$-dimensional complex $\mathrm{MB}$ with
$5$ vertices, 
\begin{figure}[h]
  \centering
\begin{equation*}
 \xy 0;/r.20pc/:
 (0,0)*{\textcolor{red}{\blacksquare}}="A"+(-4,-5)*{1};
 (40,0)*{\textcolor{red}{\blacksquare}}="B"+(0,-5)*{2};
 (80,0)*{\textcolor{red}{\blacksquare}}="C"+(0,-5)*{3};
 (120,0)*{\textcolor{blue}{\blacksquare}}="D"+(0,-5)*{4};
 (20,20)*{\textcolor{blue}{\blacksquare}}="E"+(0,5)*{4};
 (60,20)*{\textcolor{blue}{\blacksquare}}="F"+(0,5)*{5};
 (100,20)*{\textcolor{red}{\blacksquare}}="G"+(0,5)*{1};
  "A"; "B" **\dir{-};
  "A"; "E" **\dir{-};
  "B"; "E" **\dir{-};
  "B"; "F" **\dir{-};
  "E"; "F" **\dir{-};
  "B"; "C" **\dir{-};
  "E"; "F" **\dir{-};
  "C"; "F" **\dir{-};
  "F"; "G" **\dir{-};
  "E"; "G" **\dir{-};
  "D"; "G" **\dir{-};
  "C"; "G" **\dir{-};
  "C"; "D" **\dir{-};
 \endxy
\end{equation*}
    \caption{A $(\{\textcolor{blue}{\blacksquare},
    \textcolor{red}{\blacksquare}\} ,2)$-coloring of the M\"obius band}
  \label{fig:MBcol}
\end{figure}
triangulating the M\"obius band. This complex is
$(\{\textcolor{blue}{\blacksquare},
\textcolor{red}{\blacksquare}\},2)$-colorable. One needs a palette $P$ of
$5$ colors, one for each vertex, for a $(P,1)$-coloring.

The purpose of this paper is to translate colorability problems from
combinatorics to ring and homotopy theory. 

Let $\SR{K;R}$ be the Stanley--Reisner face algebra of $K$ with
coefficients in the commutative ring $R$. Also, let $R[V]$ be the
graded polynomial $R$-algebra generated by the vertex set $V$ which we
consider to be homogeneous of degree $2$.  By definition
\eqref{defn:SR} there is a surjection $R[V] \twoheadrightarrow
\SR{K;R}$.  For any subset $U$ of the vertex set $V$, let $c(U)$
denote the symmetric polynomial $\prod_{u \in U}(1+u)$ in the
polynomial ring $R[V]$ or its image in $\SR{K;R}$ under the projection
$R[V] \twoheadrightarrow \SR{K;R}$.  Also, for any natural number $s
\geq 1$, write $c_{\leq s}(U)$ for $c(U)$ truncated above degree $2s$
so that $c_{\leq s}(U)$ is the the sum of the $s$ first elementary
symmetric polynomials in the variables $U$.

\begin{thm}\label{thm:Pscol}
  The map \func fVP is a $(P,s)$-coloring of $K$ if and only if
  $c(V)= \prod_{p \in P}c_{\leq 2s}(f^{-1}p)$ in $\SR{K;R}$. 
\end{thm}

The phrasing of our main theorem uses the \daj\ space $\DaJ{K}$ and
the integral Stanley--Reisner face algebra $\SR{K;R}$.  By
construction, there is an inclusion map $\DaJ{K}
\xrightarrow{\lambda_K} \map{}V{\BU 1}$ inducing the surjection of
$H^*(\map{}V{\BU 1};R) = R[V]$ onto $H^*(\DaJ{K};R) = \SR{K;R}$.
(Readers who are not familiar with these functorial constructions can
find more information in Section~\ref{sec:DJ} and
Section~\ref{sec:SR}.)  The element $c(V)= \prod_{v \in V}(1+v)$ is
the total Chern class in $\SR{K;R}$ of the stable complex vector
bundle $\DaJ K \xrightarrow{\lambda_K} \map{}V{\BU 1}
\xrightarrow{\oplus} \mathrm{BU}$ where \func{\oplus}{\map{}V{\BU
    1}}{\mathrm{BU}} is the Whitney sum map.




\begin{thm}\label{thm:main}
 The following conditions are equivalent when the ring $R$ is a UFD:
\begin{enumerate}
\item \label{thm:main1} $K$ admits a $(P,s)$-coloring.
\item \label{thm:main5} The Stanley--Reisner ring $\SR{K;R}$ contains
  $|P|$ elements $c_p$ of degree $\deg c_p \leq 2s$, $p \in P$, so
  that $\prod_{v \in V}(1+v) = \prod_{p \in P}c_p$.
\item \label{thm:main3} There exist $|P|$ $s$-dimensional complex
  vector bundles $\xi_p$, $p \in P$, over $\DaJ K$ so that $\lambda_K$
  and $\bigoplus_{p \in P} \xi_p$ are stably isomorphic.
\item \label{thm:main4} There exists a map $\DaJ{K} \to \map{}P{\BU
    s}$ such that the diagram
\begin{equation*}
  \xymatrix{
    {\DaJ{K}} \ar[dr]_{\oplus \circ \lambda_K} \ar[rr] &&  
    {\map{}P{\BU s}} \ar[dl]^\oplus \\
    & {\mathrm{BU}} }
\end{equation*}
commutes up to homotopy.
\end{enumerate}
\end{thm}

Of course, only the cardinality of $P$ is relevant here so we shall
also say that $K$ is $(r,s)$-colorable if $K$ is $(P,s)$-colorable for
some palette $P$ of $r$ colors.

The proof of Theorem~\ref{thm:main} is in Section~\ref{sec:VBoverDJ}.
The $s=1$ version of Theorem~\ref{thm:main} appeared in \cite{DN08} by
the third author.

\subsection{Notation and the basic definition}
\label{sec:notation}
Our convention here is that any nonempty abstract simplicial complex
contains the empty set. As we are only interested in finite complexes
we shall be working with abstract simplicial complexes in the
following sense.

\begin{defn}[ASC]
  An abstract simplicial complex is a finite set of sets closed under
  formation of subsets.
\end{defn}

We shall use the following notation:
\begin{description}
\item[$|V|$] The number of elements in the finite set $V$
\item[{$D[V]$}] The ASC of {\em all\/} subsets of the finite set $V$
  ($|D[V]|=\Delta^{|V|-1}$) 
\item[{$\partial D[V]$}] The finite simplicial complex of all {\em proper}
  subsets of $V$ ($|\partial D[V]|=S^{|V|-2}$)
\item[$m_+$] For any integer $m \geq 0$, $m_+=\{0,1,\ldots,m\}$ denotes a
set of cardinality $m+1$.
\item[$m(K)$] The number of vertices $|V|$ in the ASC $K$ with vertex
  set $V$
\item[$n(K)$] The number $\max\{|\sigma| \mid \sigma \in K \}$ of
  vertices in a maximal facet of the ASC $K$
\item[$\dim K$] The dimension  $\dim K = n(K)-1$ of the ASC $K$
\item[$\codim K$] The codimension $\codim K = m(K)-n(K)$ of
  the ASC $K$ with vertex set $V$
\item[\sk jK] The $j$-skeleton $\{ \sigma \in K \mid \dim \sigma \leq j
  \}$ of the ASC $K$
  in $K$ of the simplex $\sigma
  \in K$
\item[$K_1 \ast K_2$] The join $K_1 \ast K_2 = \{\sigma_1 \amalg
  \sigma_2 \mid \sigma_1 \in K_1, \sigma_2 \in
  K_2\} \subset D[V_1 \amalg V_2]$ of ASCs $K_1$, $K_2$ with vertex
  sets $V_1$, $V_2$ 
\item[$P(K)$] The face poset of the ASC $K$ \cite[Definition~1.5]{BP02}
\item[$i_{\sigma}$] For any simplex $\sigma$ in the ASC $K$,
  $i_{\sigma} \colon \sigma \to V, D[\sigma] \to K, D[\sigma] \to
  D[V]$ are inclusion maps of sets or ASCs
\item[$S^k$] The subgroup of homogenous elements of degree $k$ in a
  $\Z$-graded ring $S=\oplus_{k \in \Z}S^k$
\item[$S^{\leq k}$] The subgroup $\oplus_{d \leq k}S^k$ of a
  $\Z$-graded ring $S=\oplus_{k \in \Z}S^k$ 
\item[$\deg(x)$] The minimal $k$ so that $x \in S^{\leq k}$ for an
  element $x$ in a $\Z$-graded ring $S=\oplus_{k \in \Z}S^k$
\end{description}

\section{\daj\ spaces}
\label{sec:DJ}

Let $V$ be a finite set and $(A,B)$ a pair of topological spaces. For
any subset $\sigma$ of $V$, write $\Map {}V{V-\sigma}AB$ for the space
of maps $(V,V-\sigma) \to (A,B)$.  When $\emptyset \subset \sigma
\subset \tau \subset V$, the inclusions $(V,V) \supset (V,V-\sigma)
\supset (V,V-\tau) \supset (V,\emptyset)$ induce inclusion maps
  \begin{equation*}
    \map{}VB = \Map{}VVAB \subset \Map{}V{V-\sigma}AB \subset
    \Map{}V{V-\tau}AB \subset \Map{}V{\emptyset}AB = \map{}VA 
  \end{equation*}
of mapping spaces. Thus we have a functor
\begin{equation}\label{eq:mapV--}
  \func {\Map {}V{V-?}AB} {P(D[V])}{\mathrm{TOP}}
\end{equation} 
from the poset of subsets of $V$ to the category of topological
spaces. There are natural transformations from the constant functor
$\map{}VB$ to this functor and from this functor to the constant
functor $\map{}VA$.

Let $K$ be an ASC with vertex set $V$.

\begin{defn} \cite[Construction~6.38]{BP02}\label{defn:DJ} 
  The \daj\ space $\DaJ{K;A,B}$ of $K$ with coefficients in $(A,B)$ is
  the colimit over the face poset $P(K)$
  \begin{equation*}
    \DaJ{K;A,B} = \colim (P(K); \Map {}V{V-?}AB) =
    \bigcup_{\sigma \in K} \Map {}V{V-\sigma}AB
  \end{equation*}
  of the functor \eqref{eq:mapV--}.
\end{defn}

The \daj\ space  $\DaJ{K;A,B}$ is born with maps 
$\map{}VB \xrightarrow{\varepsilon_K} \DaJ{K;A,B}
\xrightarrow{\lambda_K} \map{}VA$ induced by natural transformations.

In the special case where $K=D[V]$ is the full simplex, we see that
\begin{equation*}
  \DaJ {D[V];A,B} =
  \Map {}V{\emptyset}AB = \map {}VA
\end{equation*}
since the face poset $P(D[V])$ has $V$ as terminal object.

We shall just write $\DaJ{K;\BU 1,\{1\}} = \DaJ K$ in case
$(A,B)=(K(\Z,2),\{0\}) = (\BU 1,\mathrm{B}\{1\})$. This space is born
with maps $*=\map{}V* \xrightarrow{\varepsilon_K} \DaJ{K}
\xrightarrow{\lambda_K} \map{}V{\BU 1}$.




Suppose that \func fVP is a map of $V$ into some finite set $P$ and
that $(C,D)$ is a pair of topological spaces. Assume also that
$\mu=\{\mu_p \mid p \in P\}$ is a set of maps $\mu_p
\colon (\map{}{f^{-1}p}A,\map{}{f^{-1}p}B) \to (C,D)$ indexed by $P$.
Define 
\begin{equation}
  \label{eq:DJf}
  \func {\DaJ{f;\mu}}{\map{}VA}{\map{}PC}
\end{equation} 
to be the map given by
\begin{equation*}
  \DaJ{f;\mu}(\chi)(p) = \mu_p(\chi \vert f^{-1}p), \qquad 
  \chi \in \map{}VA,\  p \in P.
\end{equation*}
For instance, if $A$ is an abelian monoid with operation \func +{A
  \times A}A with neutral element $0$ and $B$ is submonoid, then we
may take \func{\mu_p} {(\map{}VA,\map{}VB)} {(A,B)} to be the map 
\begin{equation}
  \label{eq:lambdapA}
\mu_p(\chi) = \sum \chi(f^{-1}p)  
\end{equation} 
for all $\chi \in \map{}VA$.

\begin{prop}\label{prop:DJfunctor}
  $\DaJ {-;A,B}$ is a functor from the category of finite ASCs and
  injective maps to the category of spaces. If $(A,B)$ is a pair of
  abelian topological monoids then $\DaJ {-;A,B}$ is a functor from
  the category of finite ASCs to the category of spaces.
  \begin{enumerate}
  \item \label{prop:DJfunctor5} Any injective simplicial map \func
    fKL induces a map \func{\DaJ {f;A,B}}{\DaJ {K;A,B}} {\DaJ {L;A,B}}
  \item \label{prop:DJfunctor1} Any simplicial map \func fKL induces a
    map \func{\DaJ {f;A,B}}{\DaJ {K;A,B}} {\DaJ {L;A,B}} provided that
    $(A,B)$ is a pair of abelian topological monoids.
  \item \label{prop:DJfunctor4} Any simplicial map \func fKL induces a
    map \func{\DaJ {f;\mu}}{\DaJ {K;A,B}} {\DaJ {L;C,D}} provided
    that there are maps $\mu = \{\mu_p \mid p \in P\}$ indexed
    by the vertex set $P$ of $L$ as above.
\item $\DaJ{K_1 \ast K_2;A,B} = \DaJ{K_1;A,B} \times \DaJ{K_2;A,B}$
  \label{prop:DJfunctor3} 
  \end{enumerate}
\end{prop}
\begin{proof}
  \noindent \eqref{prop:DJfunctor5} Suppose that $L$ is an ASC with
  vertex set $P$ and that the injective map \func fVP is simplicial
  map $K \to L$. When $p \in P$, $v \in V$, and $p=f(v)$, take
  $\mu_p \colon (\map{}{\{v\}}A,\map{}{\{v\}}B) = (A,B) \to (A,B)$
  to be the identity map. Let $\mu=(\mu_p)$. Proceed as in the
  proof of the next item.

  \noindent \eqref{prop:DJfunctor1} Suppose that $L$ is an ASC with
  vertex set $P$ and that \func fVP is simplicial map $K \to L$. Note
  that $\DaJ{f;\mu}$ \eqref{eq:DJf} takes the subspace
  $\Map{}V{V-\sigma}AB$ into the subspace $\Map{}V{V-\tau}AB$ when
  $\sigma \in K$, $\tau \in L$ are simplices so that $f(\sigma)
  \subset \tau$ and $\mu=(\mu_p)$ is as in
  \eqref{eq:lambdapA}. Thus $\DaJ{f;\mu}$ restricts to a map
   \begin{equation*}
     \xymatrix@C=40pt{
       {\DaJ{K;A,B}} \ar[r]^{\DaJ{f;\mu}} \ar@{^(->}[d] &
       {\DaJ{L;A,B}} \ar@{^(->}[d] \\
       {\DaJ{D[V];A,B}} \ar[r]_{\DaJ{f;\mu}} &  {\DaJ{D[P];A,B}} }
   \end{equation*}
between the \daj\ spaces.

\noindent \eqref{prop:DJfunctor4} A slight generalization of the above
proof. 


\noindent \eqref{prop:DJfunctor3} \cite[Construction~6.20]{BP02}.
\end{proof}

With fixed $K$, $\DaJ{K;A,B}$ is functorial in the pair $(A,B)$.
Observe that any homotopy equivalence $(A,B) \to (A',B')$ induces a
homotopy equivalence $\DaJ{K;A,B} \to \DaJ{K;A',B'}$ because
$\Map{}V{V-\sigma}AB \to \Map{}V{V-\sigma}{A'}{B'}$ is a homotopy
equivalence for all $\sigma \in K$ and because in
Definition~\ref{defn:DJ} we may replace the colimit by the homotopy
colimit \cite[Lemma 2.7]{NR05}.

\section{The Stanley--Reisner face algebra}
\label{sec:SR}

Let $R$ be a commutative ring and let
\func{\varepsilon}PQ be a homo\m\ between 
commutative graded
$R$-algebras $P$ and $Q$.  
(A typical example will be $P=H^*(X;R)$ and $Q=H^*(A;R)$ where $(X,A)$
is a pair of spaces.)
Let $V$ be a
finite set and $\sigma$ a subset of $V$. Write
$\otimes\Map{}V{V-\sigma}PQ$ for the tensor product $\bigotimes_{v \in
  V}A_v$ of $R$-algebras
\begin{equation*}
  A_v =
  \begin{cases}
    P & v \in \sigma \\
    Q & v \not\in \sigma
  \end{cases}
\end{equation*}
for $v$ in $V$.  When $\emptyset \subset \sigma \subset \tau \subset
V$ there are \m s of $R$-algebras
\begin{equation*}
  Q^{\otimes V} = \otimes\Map{}V{V}PQ 
  \leftarrow \otimes\Map{}V{V-\sigma}PQ 
  \leftarrow \otimes\Map{}V{V-\tau}PQ
  \leftarrow \otimes\Map{}V{\emptyset}PQ = P^{\otimes V}
\end{equation*}
induced by identity and augmentation homo\m s.  For instance, if
$V=\{1,2,3\}$ and $\{1\} = \sigma \subset \tau=\{1,3\}$, then there
are \m s
\begin{equation*}
  Q \otimes_R Q \otimes_R Q 
  \xleftarrow{\varepsilon \otimes \varepsilon \otimes \varepsilon}
  P \otimes_R Q \otimes_R Q 
  \xleftarrow{1 \otimes \varepsilon \otimes \varepsilon}
  P \otimes_R Q \otimes_R P 
  \xleftarrow{1 \otimes \varepsilon \otimes 1}
  P \otimes_R P \otimes_R P 
\end{equation*}
of $R$-algebras. Thus we have a functor
\begin{equation}
  \label{eq:otimesmap}
  \otimes\Map{}V{V-?}PQ \colon P(D[V])^{\mathrm{op}} \to \mathrm{ALG}
\end{equation}
from the opposite of the poset of subsets of $V$ to the category of
graded commutative connected $R$-algebras. There are natural
transformations from this functor to the constant functor $Q^{\otimes
  V}$ and from the constant functor $P^{\otimes V}$ to this functor.

Let $K$ be a finite ASC with vertex set $V$.

\begin{defn}\label{defn:SR} 
  The {\em Stanley--Reisner algebra of $K$ with coefficients in
    $(P,Q)$\/} is the limit over the opposite face poset
  $P(K)^{\mathrm{op}}$
  \begin{equation*}
    \SR{K;P,Q} = {\lim}(P(K)^{\mathrm{op}};\otimes\Map{}V{V-?}PQ)
  \end{equation*}
  of the functor \eqref{eq:otimesmap}.
\end{defn}

The Stanley--Reisner algebra $\SR{K;P,Q}$ is born with $R$-algebra
homo\m s $Q^{\otimes V} \leftarrow \SR{K;P,Q} \leftarrow P^{\otimes
  V}$ induced by natural transformations.

In the special case where $K=D[V]$ is the full simplex, we see that
\begin{equation*}
  \SR{D[V];P,Q} =
  \otimes\Map {}V{\emptyset}PQ = P^{\otimes V}
\end{equation*}
since the face poset $P(D[V])$ has $V$ as terminal object.

For any simplicial map \func fKL between two finite ASCs, let
\func{\SR{f}}{\SR{L;P,Q}}{\SR{K;P,Q}} be the $R$-algebra homo\m\
induced by the poset map \func{P(f)}{P(K)}{P(L)}. This makes
$\SR{-;P,Q}$ into a contravariant functor from the category of finite
ASCs into the category of graded commutative $R$-algebras.

\begin{prop}\cite[Theorem 3.10]{NR05}\label{prop:HDJR}
  $H^*(\DaJ{K;X,*};R) = \SR{K;H^*(X;R),H^*(*;R)}$ for any CW-complex
  $X$ with base point $*$.
\end{prop}

The case where $X=\BU 1 = \C P^{\infty}$ is the classifying space for
complex line bundles will be especially relevant for this paper.  We
shall abbreviate $\DaJ{K;\BU 1,*}$ to $\DaJ K$ and $\SR{K;H^*(\BU
  1;R),H^*(*;R)}$ to $\SR{K;R}$. The inclusion \func{i_K}K{D[V]} of
$K$ into the full simplex on its vertex set induces an inclusion of
spaces
\begin{equation*}
\func{\DaJ{i_K}}{\DaJ K}{\DaJ{D[V]} = \BU 1^V}  
\end{equation*}
inducing the algebra \m\
\begin{equation*}
  \SR{K;R} = H^*(\DaJ K;R) \leftarrow H^*(\DaJ{D[V]};R) =
  H^*(\BU 1;R)^{\otimes V} = 
  R[V]
\end{equation*}
where $R[V]$ is the polynomial ring generated by $V$ in degree $2$.

For any subset $U$ of $V$, write $(U) \subset R[V]$ for the ideal
generated by the set $U \subset R[V]$ and $\prod U, \sum U \in R[V]$
for the elements $\prod_{u \in U}u$, $\sum_{u \in U}u$.

\begin{prop}\cite[(4.7)]{PRV04}\label{prop:SRideal}
  $\SR{K;R}$ is isomorphic to the quotient of $R[V]$ by the
  ideal $\left( \prod \tau \mid \tau \in D[V]- K
  \right)$.
\end{prop}
\begin{proof}
  The homo\m\ $\SR{i_K} \colon R[V] \to \SR{K;R}$ is quickly seen to
  be surjective.  The kernel is the ideal $\bigcap_{\sigma \in K}
  \left(V-\sigma\right)$ which equals \cite[Proposition 3.6]{NR05} the
  ideal generated by the set $\{\prod \tau \mid \tau \in D[V]-K \}
  \subset R[V]$.
\end{proof}

The ideal $\left( \prod \tau \mid \tau \in D[V]- K \right)$ of
Proposition~\ref{prop:SRideal} is called the {\em Stanley--Reisner
  ideal}.



Define $c(V) \in R[V]$ by
\begin{equation}\label{def:LV} c(V) = \prod_{v \in
    V}(1+v)
\end{equation} 
For any map \func fVP, the induced map \func{R[f]}{R[P]}{R[V]} takes
$c(P)$ in $R[P]$ to 
\begin{equation}\label{eq:Rfc}
  R[f]c(P) = \prod_{p \in P}(1+\sum f^{-1}p)
\end{equation} 
in $R[V]$. 
\begin{lemma}\label{lemma:RfC}
$\text{\func fVP is injective} \iff R[f] c(P) = c(V)$  
\end{lemma}
\begin{proof}
  If \func fVP is injective, then \eqref{eq:Rfc} shows that $
  R[f] c(P) = c(V)$ because $f^{-1}(p)$ is nonempty for
  exactly $|V|$ points of $P$ and for these points $f^{-1}(p)$ is a
  single point of $V$. 

  Conversely, notice that $\deg c(V) =2|V|$ and $\deg R[f] c(P) =
  2|f(V)|$.  If $R[f] c(P) = c(V)$ then $|V|=|f(V)|$ and $f$ is
  injective.
\end{proof}



Define  
\begin{equation}\label{def:LK}
c(\lambda_K) = \SR{i_K}c(V)
\end{equation}
to be the image in $\SR{K;R}$ of $c(V)$ in $R[V]$. (Then $c(V) =
c(\lambda_{D[V]})$ in $R[V] = \SR{D[V];R}$.)

\begin{lemma}\label{lemma:charcK}
  Two elements of $\SR{K;R}$ are identical if and only if they have
  identical $\SR{i_{\sigma}}$-images in $\SR{D[\sigma];R} = R[\sigma]$
  for all $\sigma \in K$. In particular, for any $\sigma \in K$,
  $\SR{i_{\sigma}} c(\lambda_K) = c(\sigma)$ in $\SR{D[\sigma];R} =
  R[\sigma]$, and $c(\lambda_K)$ is the only element of $\SR{K;R}$
  with this property.
\end{lemma}
\begin{proof}
  $\SR{i_{\sigma}} c(\lambda_K) = \SR{i_{\sigma}} \SR{i_K}c(V) =
  \SR{i_K \circ i_{\sigma}}c(V) = \SR{i_{\sigma}}c(V) = c(\sigma)$.
  Since $\SR{K;R}$ is a subring of the product of the rings
  $\SR{D[\sigma];R}=R[\sigma]$, $\sigma \in K$, this property
  characterizes $c(\lambda_K)$.
\end{proof}

 
The element $c(\lambda_K)$ of the graded ring $\SR{K;R}$ has degree
$\deg c(\lambda_K) = n(K)$.

\section{Splittings of vector bundles over \daj\ spaces}
\label{sec:unstablesplit}

Let $\VU 1 \to \BU 1$ denote the universal and $\BU 1 \times \C \to
\BU 1$ the trivial complex line bundle.

\begin{lemma}\label{lemma:VU1C}
  There exist a vector bundle map $ \VU 1 \xrightarrow a \BU 1 \times
  \C$ and a contractible subspace $\BU 1^\times$ of $\BU 1$ such that
  $a \mid \VU 1^\times$ trivializes the restriction $\VU 1^\times \to
  \BU 1^\times$ of the universal line bundle to $\BU 1^\times$.
\end{lemma}
\begin{proof}
  Let $\C(1)$ and $\C(0)$ be $\C$ with the standard and the trivial $\U
  1$-actions, respectively.
  The projection
  $\mathrm{pr}_1 \colon \C(1) \times \C(1) \to \C(1)$ onto the first factor
  and \func a{\C(1) \times \C(1)}{\C(0)}
  given by $a(x,y)=\overline{x}y$, where
  $\overline{x}$ is the complex conjugate of $x$, are $\U
  1$-maps. They determine commutative diagrams
  \begin{equation*}
    \xymatrix{
      {\C(1) \times \C(1)} \ar[rr]^{\mathrm{pr}_1 \times a} 
      \ar[dr]_{\mathrm{pr}_1} &&
      {\C(1) \times \C(0)} \ar[dl]^{\mathrm{pr}_1} &
      {\VU 1} \ar[dr] \ar[rr]^-a &&
      {\BU 1 \times \C} \ar[dl]\\
      & {\C(1)} &&& {\BU 1} }
  \end{equation*}
  where the right diagram is obtained from the left diagram of $\U
  1$-spaces and $\U 1$-maps by applying the homotopy orbit space
  functor (Definition~\ref{defn:htpyorbit}). Here we use that the
  projection \func{\mathrm{pr}_1}{\C(1) \times \C(1)}{\C(1)} induces
  the universal line bundle $\VU 1 = (\C(1) \times \C(1))_{h\U 1} \to
  \C(1)_{h\U 1} = \BU 1$.  Let $\C(1)^\times$ be the $\U 1$-subspace
  of nonzero elements of $\C(1)$. Observe that the restriction of the
  vector bundle map $\VU 1 \to \BU 1 \times \C$ trivializes the vector
  bundle $(\C(1)^\times \times \C(1))_{h\U 1} \to \C(1)^\times_{h\U
    1}$ over the contractible subspace $\BU 1^\times =
  \C(1)^\times_{h\U 1}$ of $\BU 1 = \C(1)_{h\U 1}$.
\end{proof}

\begin{defn}\label{defn:htpyorbit}
  Let $G$ be a compact Lie group and $X$ a left $G$-space. The
  homotopy orbit space $X_{hG}$ is the space $(EG \times X)/G$ of
  $G$-orbits for the $G$-action $(e,x) \cdot g = (eg,g^{-1}x)$, $e \in
  EG$, $x \in X$, $g \in G$, on $EG \times X$ where $EG$ is a
  contractible free right $G$-space.
\end{defn}

By abuse of notation, we shall also let $\lambda_K$ stand for the
$|V|$-dimensional complex vector bundle 
\begin{equation}
  \label{eq:lambdaxi}
  \xymatrix@1{
    {\map{}V{\C}} \ar[r]^-{\varepsilon_K} &
    {\DaJ{K;\VU 1,\C}} \ar[r] & 
    {\DaJ{K; \BU 1, *}}}
\end{equation} 
over $\DaJ K = \DaJ{K;\BU 1,*}$ classified by the map $\DaJ{K; \BU 1,
  *} \xrightarrow{\lambda_K} \map{}V{\BU 1} \xrightarrow \oplus
\BU{|V|}$.

\begin{lemma}\label{lemma:xi}\cite[Theorem 1.1]{DN09}
  There exists a fibrewise surjective vector bundle map
  \begin{equation*}
    \xymatrix{
    {\DaJ{K;\VU 1,\C}} \ar[dr] \ar[rr] &&
    { \C^{\codim K} \times \DaJ K}
    \ar[dl] \\
    & {\DaJ K}}
  \end{equation*}
  of $\lambda_K$ to the trivial vector bundle of dimension $\codim K =
  m(K)-n(K)$ over $\DaJ K$.
\end{lemma}
\begin{proof}
  Let $a \colon \VU 1 \to \BU 1 \times \C$ be the vector bundle map
  from Lemma~\ref{lemma:VU1C}. Choose a linear map \func
  A{\map{}V{\C}}{\C^{\codim K}} such that the restrictions of $A$ to
  the subspaces $\Map{}V{\sigma}{\C}0$, $\sigma \in K$, are surjective
  for all $\sigma \in K$. This is possible since $\dim
  \Map{}V{\sigma}{\C}0 = |V| - |\sigma| \geq |V| - \max\{|\sigma| \mid
  \sigma \in K\} = \codim K$. Then the vector bundle map
  \begin{equation*}
    {\map{}V{\VU 1}} \xrightarrow{\map{}Va}  
    {\map{}V{\BU 1} \times \map{}V{\C}} \xrightarrow{\mathrm{id} \times A} 
    {\map{}V{\BU 1} \times \C^{\codim K}} 
  \end{equation*}
  of vector bundles over $\map{}V{\BU 1}$ restricts to a surjective
  vector bundle map
  \begin{equation*}
   \DaJ{K;\VU 1,\VU 1^\times} \to 
   \DaJ{K;\BU 1,\BU 1^\times} \times \C^{\codim K}    
  \end{equation*}
  of vector bundles over $\DaJ{K;\BU 1,\BU 1^\times}$ because there
  are commutative diagrams
  \begin{equation*}
    \xymatrix{
      {\map{}{V,\sigma}{\C,0}} \ar@{=}[r] \ar@{^(->}[d] &
      {\map{}{V,\sigma}{\C,0}} \ar@{^(->}[dd]\ar[ddr]^\cong \\
      {\map{}V{\C}} \ar@{^(->}[d]  \\
      {\map{}{V,V-\sigma}{\VU 1,\VU 1^\times}} \ar[d] \ar[r]^-{a} &
      {\map{}V{\BU 1} \times \map{}V{\C}} \ar[r]^-{A \circ \mathrm{pr}_2} &
      {\C^{\codim K}}\\
      {\map{}{V,V-\sigma}{\BU 1,\BU 1^\times}} }
  \end{equation*}
  over $\map{}{V,V-\sigma}{\BU 1,\BU 1^\times} \subset \DaJ{K;\BU
    1,\BU 1^\times}$ for every simplex $\sigma \in K$.

  The inclusion map $\DaJ{K;\BU 1,*} \to \DaJ{K;\BU 1,\BU 1^\times}$
  is a homotopy equivalence covered by the vector bundle map
  $\DaJ{K;\VU 1,\C} \to \DaJ{K;\VU 1,\VU 1^\times}$.
\end{proof}

Let $\xi_K$ be the kernel of the vector bundle \m\ from
Lemma~\ref{lemma:xi} so that the vector bundles $\lambda_K$ and
$\xi_K$ are related by a short exact sequence
 \begin{equation}\label{eq:fconst} \xymatrix@1{ 0 \ar[r] & \xi_K
    \ar[r] & \lambda_K \ar[r] & \C^{\codim K} \ar[r] & 0}
\end{equation} 
of complex vector bundles over $\DaJ K$, and  $\dim \xi_K = \max\{
|\sigma| \mid \sigma \in K \} = n(K)$. 

Let \func fVP be any map. For each $p \in P$, define
\begin{equation*}
  d_f(p) = \max
\{ |\sigma \cap f^{-1}(p) | \mid \sigma \in K \}
\end{equation*}
to be the maximal number of vertices of color $p$ in any simplex of
$K$. Then
 \begin{equation}\label{eq:colorineq}
  n(K) \leq \sum_{p \in P} d_f(p) \leq m(K)
\end{equation} 
and \func fVP is a $(P,s)$-coloring if and only if $d_f(p) \leq
s$ for all colors $p \in P$.  For each color $p \in P$, consider the
subcomplex of $K$ consisting of all monochrome simplices of color $p$,
\begin{equation*}
  K_p = K \cap D[f^{-1}p] = \{ \sigma \in K \mid f(\sigma)=p \},
\end{equation*}
of dimension $\dim K_p = d_f(p)-1$ and codimension $\codim K_p =
|f^{-1}(p)|-d_f(p)$.  Since $V = \coprod_{p \in P}f^{-1}p$ is the
disjoint union of monochrome subsets there is an injective simplicial
map \func{d(f)}{K}{\bigstar_{p \in P}K_p} of $K$ into the join of the
subcomplexes $K_p$ inducing
(Proposition~\ref{prop:DJfunctor}.\eqref{prop:DJfunctor3}) a map
 \begin{equation}
  \label{eq:DJdf}
  \func{\DaJ{d(f)}}{\DaJ{K;A,B}}{\DaJ{\bigstar_{p \in P}K_p;A,B} = \prod_{p \in
      P}\DaJ{K_p;A,B}}
\end{equation} 
of $\DaJ K$ into the product of the $\DaJ {K_p}$, $p \in P$.  Write
$\xi_p$ and $\lambda_p$ for the vector bundles $\xi_{K_p}$ and
$\lambda_{K_p}$ over $\DaJ{K_p}$ and for their pull backs to $\DaJ K$.
We have $\dim \xi_p = d_f(p)$ and $\dim \lambda_p = |f^{-1}p|$.

 \begin{thm}\label{thm:fsplit}
   Associated to any map \func fVP, there is a short exact sequence
  \begin{equation*}
  \xymatrix@1{
    0 \ar[r] &
    {\bigoplus_{p \in P} \xi_p} \ar[r] &
    {\lambda_K} \ar[r] &
    {\C^{m(K)-\sum d_f(p)}} \ar[r] &
    0}
\end{equation*}
of vector bundles over $\DaJ K$ where $\dim \xi_p = d_f(p)$ and
$c(\xi_p)=c(f^{-1}p)$.
 \end{thm}
 \begin{proof}
   If we pull back the short exact sequence
\begin{equation*}
  \xymatrix@1{
    0 \ar[r] &
    {\prod_{p \in P} \xi_p} \ar[r] &
    {\prod_{p \in P} \lambda_p} \ar[r] &
    {\prod_{p \in P} \C^{|f^{-1}p|-d_f(p)}} \ar[r] &
    0 }
\end{equation*}  
of vector bundles over $\prod_{p \in P}\DaJ{K_p}$ along the map 
\eqref{eq:DJdf}, we obtain a short exact sequence
\begin{equation*}
  \xymatrix@1{
    0 \ar[r] &
    {\bigoplus_{p \in P} \xi_p} \ar[r] &
    {\lambda_K} \ar[r] &
    {\C^{m(K)-\sum d_f(p)}} \ar[r] &
    0}
\end{equation*}
of vector bundles of $\DaJ K$. Here, we use that $\lambda_K =
\bigoplus_{p \in P} \lambda_p$ as shown by the pull-back
\begin{equation*}
  \xymatrix@C=45pt{
    {\DaJ{K;\VU 1,\C}} \ar[r]^-{\DaJ{d(f)}} \ar[d]_{\lambda_K} &
    {\prod_{p \in P}\DaJ{K_p;\VU 1,\C}} \ar[d]^{\prod \lambda_p} \\
    {\DaJ{K;\BU 1,*}} \ar[r]_-{\DaJ{d(f)}} &
    {\prod_{p \in P}\DaJ{K_p;\BU 1,*}}} 
\end{equation*}
of $|V| = \sum |f^{-1}p|$-dimensional vector bundles.
\end{proof}
 
The isomorphism class of the vector bundle $\xi_K$ is independent of
the choice of $\map{}V{\C} \xrightarrow A \C^{|V|}/\C^{1+\dim K}$
because any two choices are identical up to base changes. One
possibility is to take $A$ to be the linear map whose associated
$(\codim K \times |V|)$-matrix $A=(i^{j-1})$, $1 \leq i \leq \codim
K$, $1 \leq j \leq |V|$, is a Vandermonde matrix \cite[III, \S 8, no.\
6]{bourbaki98}. In the situation of Theorem~\ref{thm:fsplit} we have a
commutative diagram
\begin{equation*}
  \xymatrix@C=40pt{
    {\map{}{\bigcup_{p \in P}f^{-1}p}{\C}} \ar@{=}[d] \ar[r]^-A &
    {\C^{\sum |f^{-1}p|}/\C^{n(K)} = \C^{\codim K}} \ar@{->>}[d] \\
    {\prod_{p \in P} \map{}{f^{-1}p}{\C}} \ar[r]_-{\prod_{p \in P}A_p} &
    {\prod_{p \in P} \big(\C^{|f^{-1}p|}/\C^{n(K_p)}\big) = 
      \prod_{p \in P}\C^{\codim K_p}} }
\end{equation*}
where the right vertical epi\m\ is induced by an inclusion of
$\C^{n(K)}$ into $\prod_{p \in P} \C^{n(K_p)}$ which exists
because $n(K) \leq \sum_{p \in P}n(K_p)$ or, equivalently,
$\codim K \geq \sum_{p \in P} \codim K_p$.  The $(\codim K_p \times
|f^{-1}p|)$-matrix for $A_p$ is the submatrix of $A$ consisting of the
first $\codim K_p$ rows and the columns corresponding to the subset
$f^{-1}p$ of $V$. This is also a Vandermonde matrix so that the linear
maps $A_p$, $p \in P$, satisfy the condition from the proof of
Lemma~\ref{lemma:xi}. Therefore the above diagram induces another
commutative diagram
\begin{equation*}
  \xymatrix{
    0 \ar[r] &
    {\xi_K} \ar[r] \ar[d] &
    {\lambda_K} \ar@{=}[d] \ar[r] &
    {\C^{\codim K}} \ar[d] \ar[r] &
    0 \\
    0 \ar[r] &
    {\bigoplus_{p \in P} \xi_p} \ar[r] &
    {\bigoplus_{p \in P} \lambda_p} \ar[r] &
    {\bigoplus_{p \in P} \C^{\codim K_p}} \ar[r] &
    0 }
\end{equation*}
of vector bundle \m s from which we get the short exact sequence
\begin{equation}\label{eq:xisumpxi}
  \xymatrix@1{
    0 \ar[r] & 
    {\xi_K} \ar[r] &
    {\bigoplus_{p \in P} \xi_p} \ar[r] & 
    {\C^{\codim K - \sum \codim K_p}} \ar[r] &
    0 }
\end{equation}
of vector bundles over $\DaJ K$.

\section{Vertex colorings of simplicial complexes}
\label{sec:ideal}

Let $K$ be a finite ASC with vertex set $V$ and 
$R$ a commutative ring.


\begin{defn}
  Let $P$ be a finite set, a palette of colors.
  \begin{itemize}
  \item A $P$-coloring of $K$ is a map \func{f}VP that restricts to
    injective maps $f
    \vert \sigma \colon \sigma \to P$ 
    on all simplices $\sigma \in K$.
  \item $K$ is $r$-colorable if $K$ admits a coloring from a palette
    $P$ of $r$ colors.
  \item The chromatic number of $K$, $\ch K$, is the least $r$ so that
  $K$ is $r$-colorable.
  \end{itemize}
 \end{defn}

The identity map \func{1_V}VV is a $V$-coloring of $K$ painting the
vertices with distinct colors. 
The chromatic number of the full simplex $D[V]$ is $\ch {D[V]} = |V|$. If
$K'$ is a subcomplex of $K$ then $\ch {K} \geq \ch {K'}$. In particular,
$|\sigma| = \ch{D[\sigma]} \leq \ch K \leq \ch{D[V]} = |V|$ so that
\begin{equation*}
  n(K)\leq \ch K \leq m(K)   
\end{equation*}
for any ASC $K$ with vertex set $V$.

Let $\mathrm{Col}(K,P)$ be the set of $P$-colorings of $K$. Then 
\begin{equation*}
\mathrm{Col}(K,P) = {\lim}(P(K)^{\mathrm{op}};\mathrm{map}^1(-,P))  
\end{equation*}
is the limit of the set-valued functor $\mathrm{map}^1(-,P) \colon
P(K)^{\mathrm{op}} \to \mathrm{SET}$, taking $\sigma \in K$ to the set
$\mathrm{map}^1(\sigma,P)$ of injective maps $\sigma \to P$.  Since
$\mathrm{map}^1(\sigma,P) = \mathrm{Col}(D[\sigma],P)$, because the
face poset of the full simplex $D[\sigma]$ has $\sigma$ as a final
element, we could also write
\begin{equation*}
  \mathrm{Col}(K,P) = {\lim}(P(K)^{\mathrm{op}};\mathrm{Col}(D[-],P))
\end{equation*}
to emphasize that a global $P$-coloring of $K$ is a coherent choice of
local $P$-colorings of its simplices.

For any injective simplicial map \func fKL between two finite ASCs,
$\func{\mathrm{Col}^1(f,P)}{\mathrm{Col}(L,P)}{\mathrm{Col}(K,P)}$
is the map induced by the injective poset map \func{P(f)}{P(K)}{P(L)}.
Thus $\mathrm{Col}(-,P)$ is a contravariant set-valued functor on
the category of finite ASCs with injective maps.

\begin{defn}\cite[Definition~2.20]{BP02}\label{defn:1flag}
The {\em flagification\/}  of $K$ is the ASC
\begin{equation*}
  \fla {}K = \{ \sigma \in D[V] \mid \sk 1{D[\sigma]}
  \subset K \}
\end{equation*}
and $K$ is a flag complex if $K$ equals its flagification.
\end{defn}

We have $\sk 1K \subset K \subset \fla {}K \subset D[V]$ and $\sk 1K =
\sk 1{\fla {}K}$; $\fla {}K$ is the largest subcomplex of $D[V]$ with
the same $1$-skeleton as $K$.
  
The {\em missing faces\/} of $K$ are the minimal elements of the poset
$D[V]-K$ \cite[Definition~2.21]{BP02}; they generate the Stanley--Reisner
ideal.

The following two propositions emphasize that coloring issues are
$1$-dimensional.

\begin{prop}\cite[Proposition~2.22]{BP02}\label{prop:1flag}
  The following conditions are equivalent:
  \begin{enumerate}
  \item $K$ is flag
  \item $\forall \sigma \in D[V] \colon \sk 1{D[\sigma]} \subset K
    \Longrightarrow \sigma \in K$
  \item The missing faces of $K$ are $1$-dimensional
  \end{enumerate}
\end{prop}

\begin{prop}\label{prop:1col}
  The following conditions are equivalent for any map \func fVP:
  \begin{enumerate}
  \item $f$ is a $P$-coloring of $K$
  \item $f$ is a $P$-coloring of $\sk 1K$
  \item $f$ is a $P$-coloring of $\fla {}K$
  \end{enumerate}
  Moreover, $\mathrm{Col}(K,P) = \mathrm{Col}(\sk 1K,P) =
  \mathrm{Col}(\fla {}K,P)$ and $\ch K = \ch{\sk 1K} =
  \ch{\fla {}K}$.  
\end{prop}

\begin{thm}\label{thm:RKcol}
  The map \func{f}VP is a $P$-coloring of $K$ if and only if $
  c(\lambda_K) = \SR f c(P)$ in $\SR{K;R}$.
\end{thm}
\begin{proof}
  We have
  \begin{align*}
    \text{\func fVP is a coloring} &\iff
    \forall \sigma \in K \colon 
    \text{\func{f \circ i_{\sigma}}{\sigma}P is injective} \\
    &\iff  \forall \sigma \in K \colon   c(\sigma) = \SR{f \circ
      i_{\sigma}}c(P) \\
    &\iff  \forall \sigma \in K \colon
    \SR{i_{\sigma}} \big( c(\lambda_K) \big) = 
    \SR{i_{\sigma}} \big(\SR{f}c(P) \big)  \\
    &\iff c(\lambda_K) = \SR{f}c(P) 
  \end{align*}
  from Lemma~\ref{lemma:RfC} and \ref{lemma:charcK}. (We here regard
  \func fVP as the simplicial map $K \stackrel{i_K}{\subset} D[V]
  \xrightarrow{D[f]} D[P]$.)
\end{proof}

In other words, a partition $V=V_1 \cup \cdots \cup V_r$ of $V$ into
  $r$ disjoint nonempty subsets is an $r$-coloring of $K$ if and only
  if the equation
  \begin{equation*}
   \prod_{v \in V}(1+v) = \prod_{1 \leq j \leq r}(1+\sum V_j)
  \end{equation*}
holds in the  Stanley--Reisner algebra $\SR{K;R}$.


  \begin{exmp}\label{exmp:C5}
    The cyclic simplicial graph $C_5$ on the $5$ vertices in
    $V=\{v_1,\ldots,v_5\}$ has Stanley--Reisner ring
    \begin{equation*}
      \SR{C_5;\Z} = \Z[v_1,\ldots,v_5]/(v_1v_3,v_1v_4,v_2v_4,v_2v_5,v_3v_5)
    \end{equation*}
    The chromatic number number $\ch {C_5}=3$ and there are $30$
    different $3$-colorings of $C_5$. Since
    $[v_1,v_2,v_3,v_4,v_5] \to [1,2,1,2,3]$ is a coloring, the
    identity
  \begin{equation*}
    \prod_{1 \leq i \leq 5}(1+v_i) = 
    (1+\textcolor{blue}{v_1+v_3})
    (1+\textcolor{red}{v_2+v_4})
    (1+\textcolor{green}{v_5})
  \end{equation*}
  holds in $\SR{C_5;\Z}$. $C_5$ is a flag complex.
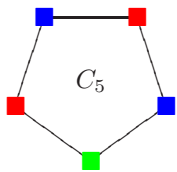
\begin{equation*}
 \xy /l2.5pc/:
   {\xypolygon5"A"{}};
   "A3"*{\textcolor{blue}{\blacksquare}}; 
   "A4"*{\textcolor{red}{\blacksquare}}; 
   "A5"*{\textcolor{blue}{\blacksquare}};
   "A1"*{\textcolor{red}{\blacksquare}};
   "A2"*{\textcolor{green}{\blacksquare}}; 
   "A0"*{C_5};
\endxy
\end{equation*}
  \end{exmp}

  \begin{exmp}\label{exmp:41coldisc}
    A triangulation $K=\{\{1,2,4\},\{1,3,4\},\{1,2,3\}\} = D[0_+] *
    \partial D[2_+]$ of the $2$-disc with chromatic number
    $4$
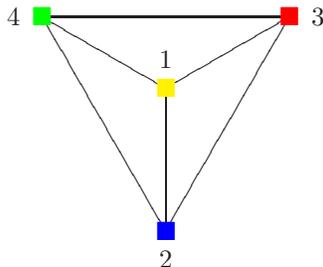
\begin{equation*}
       \xy /l4.5pc/:
         {\xypolygon3"A"{~<{-}}};
          "A1"*{\textcolor{blue}{\blacksquare}}+(0,0.2)*{2}; 
          "A2"*{\textcolor{red}{\blacksquare}}+(-0.2,0)*{3}; 
          "A3"*{\textcolor{green}{\blacksquare}}+(0.2,0)*{4}; 
          "A0"*{\textcolor{yellow}{\blacksquare}}+(0,-0.2)*{1}; 
        \endxy
    \end{equation*}
    and Stanley--Reisner algebra $\SR{K;\Z} =
    \Z[v_1,v_2,v_3,v_4]/(v_2v_3v_4) = \Z[v_1] \otimes
    \Z[v_2,v_3,v_4]/(v_2v_3v_4)$.  
  \end{exmp}

  \section{Relaxed vertex colorings of simplicial complexes}
  \label{sec:gencol}

  We introduce vertex colorings that allow a bounded number of
  vertices in any simplex to have the same color.



\begin{defn}\label{defn:scol}
  Let $P$ be a finite set (a palette of colors) and $s$ a natural
  number where $1 \leq s \leq \dim(K)$.
  \begin{itemize}
  \item A $(P,s)$-coloring  of $K$  is a map \func{f}VP such that
    every simplex of $K$ contains at most $s$ vertices of the same
    color: $\forall \sigma \in K \forall p \in P \colon |\sigma \cap
    f^{-1}(p)| \leq s$.
  \item $K$ is $(r,s)$-colorable if $K$ admits a $(P,s)$-coloring
    from a palette $P$ of $r$ colors.
  \item The $s$-chromatic number of $K$, $\chs sK$, is the least $r$ so that
  $K$ is $(r,s)$-colorable.
  \end{itemize}
  \end{defn}
 
  In other words, the vertex coloring \func fVP is a $(P,s)$-coloring
  if $K$ has no monochrome $s$-dimensional simplices.
  A $(P,1)$-coloring of $K$ is a standard coloring of $K$, and $\ch
  K = \chs 1K$. Clearly,
  \begin{equation*}
    m(K) \geq \chs 1K  \geq \cdots \geq \chs sK \geq \chs
    {s+1}K \geq \cdots \geq \chs {n(K)}K = 1 
  \end{equation*}
  as any $(P,s)$-coloring is also a $(P,s+1)$-coloring.
  The map \func fVP is a $(P,s)$-coloring of $K$ if and
  only if $|f(\sigma)|>1$ for all $s$-dimensional simplices $\sigma
  \in K$, ie if and only if $K$ contains no monochrome $s$-dimensional
  simplices.  If \func {f_1}V{P_1} is a $(P_1,s_1)$-coloring and \func
  {f_2}{P_1}{P_2} is $s_2$-to-$1$ then $f_2 \circ f_1$ is an
  $s_1s_2$-coloring. Any finite ASC with vertex set $V$ is $(\lceil
  |V|/s \rceil,s)$- and $(r, \lceil |V|/r \rceil)$- colorable. The
  $s$-chromatic number of a full simplex is $\chs s{D[V]} = \lceil
  |V|/s \rceil$. (For a real number $t$, $\lceil t \rceil$ denotes the
  least integer $\geq t$.)


Let $\mathrm{Col}^s(K,P)$ be the set of $(P,s)$-colorings of $K$. Then 
\begin{equation*}
\mathrm{Col}^s(K,P) = {\lim}(P(K)^{\mathrm{op}};\mathrm{map}^s(-,P))  
\end{equation*}
is the limit of the set-valued functor $\mathrm{map}^s(-,P) \colon
P(K)^{\mathrm{op}} \to \mathrm{SET}$, taking $\sigma \in K$ to the set
$\mathrm{map}^s(\sigma,P)$ of $s$-to-$1$ maps $\sigma \to P$.  In
particular, $\mathrm{map}^s(\sigma,P) = \mathrm{Col}^s(D[\sigma],P)$
because the face poset of the full simplex $D[\sigma]$ has $\sigma$ as
a final element, and we could also write
\begin{equation*}
  \mathrm{Col}^s(K,P) = {\lim}(P(K)^{\mathrm{op}};\mathrm{Col}^s(D[-],P))
\end{equation*}
to emphasize that a global $(P,s)$-coloring of $K$ is a coherent choice of
local $(P,s)$-colorings of its simplices.

For any injective simplicial map \func fKL between two finite ASCs,
$\func{\mathrm{Col}^s(f,P)}{\mathrm{Col}^s(L,P)}{\mathrm{Col}^s(K,P)}$
is the map induced by the injective poset map \func{P(f)}{P(K)}{P(L)}.
Thus $\mathrm{Col}^s(-,P)$ is a contravariant set-valued functor on
the category of finite ASCs with injective maps.


In particular, $\mathrm{map}^s(\sigma,P) =
\mathrm{Col}^s(D[\sigma],P)$ because the face poset of the full
simplex $D[\sigma]$ has $\sigma$ as a final element, and we could also write
\begin{equation*}
  \mathrm{Col}^s(K,P) = {\lim}(P(K)^{\mathrm{op}};\mathrm{Col}^s(D[-],P))
\end{equation*}
to emphasize that a global $(P,s)$-coloring of $K$ is a coherent choice
of local $(P,s)$-colorings of its simplices.

If $K$ contains $K'$ as a subcomplex then there is a restriction map
$\mathrm{Col}^s(K,P) \to \mathrm{Col}^s(K',P)$ and $\chs sK \geq \chs
s{K'}$. In particular, $\lceil |\sigma|/s \rceil = \chs s{D[\sigma]}
\leq \chs sK \leq \chs s{D[V]} = \lceil |V|/s \rceil$ whenever $\sigma
\in K$ so that
  \begin{equation}\label{eq:chrs}
    \left\lceil \frac{n(K)}s \right\rceil 
    \leq \chs sK 
    \leq \left\lceil \frac{m(K)}s \right\rceil
  \end{equation} 
for any ASC $K$.

Since any $s$-to-$1$ map is an $(s+1)$-to-$1$ map there are inclusions
\begin{equation*}
   \mathrm{map}^1(\sigma,P) \subset \cdots \subset
 \mathrm{map}^s(\sigma,P) \subset 
\mathrm{map}^{s+1}(\sigma,P) \subset 
  \cdots \subset \mathrm{map}^{|\sigma|}(\sigma,P) =  \map {}{\sigma}P 
\end{equation*}
that produce inclusions
\begin{equation*}
   \mathrm{Col}^1(K,P) \subset \cdots \subset \mathrm{Col}^s(K,P)
   \subset \mathrm{Col}^{s+1}(K,P) \subset \cdots \subset 
   \mathrm{Col}^{\dim(K)+1}(K,P) = \map {}VP
\end{equation*}
giving $ \mathrm{Col}^*(K,P)$ the structure of a filtered set.

For any subset $U$ of the vertex set $V$, let $c_{\leq s}(U) \in R[V]$
denote the sum of the $s$ first elementary symmetric polynomials in
the elements of $U$. For instance,
\begin{equation*}
  c_{\leq 0}(U) = 1 =c_{\leq 0}(\emptyset) , \quad
  c_{\leq 1}(U) = 1+\sum_{u \in U}u, \quad
  c_{\leq 2}(U) = 1+\sum_{u \in U}u + \sum_{u_1 \neq u_2}u_1u_2,
  \ldots,
   c_{\leq |U|}(U) = \prod_{u \in U}(1+u)
\end{equation*}
Lemma~\ref{lemma:RfC} says that \func fVP is $1$-to-$1$ if and only if
$c(V) = \prod_{p \in P} c_{\leq 1}(f^{-1}p)$ in $R[V]$ and
Theorem~\ref{thm:RKcol} says that \func fVP is a $(P,1)$-coloring of
$K$ if and only if $c(\lambda_K)= \prod_{p \in P}c_{\leq 1}(f^{-1}p)$
in $\SR{K;R}$.  We can now obtain more general statements.

\begin{lemma}\label{lemma:sto1inRV}
  $\text{\func fVP is $s$-to-$1$} \iff c(V) = \prod_{p \in
    P} c_{\leq s}(f^{-1}p)$
\end{lemma}
\begin{proof}
  If \func fVP is $s$-to-$1$ then $c_{\leq s}(f^{-1}p)$ is the total
  Chern class of the set $f^{-1}p$ so that $c_{\leq s}(f^{-1}p) =
  \prod_{f(v)=p}(1+v)$. Therefore, $c(V) = \prod_{v \in
    V}(1+v) = \prod_{p \in P}\prod_{f(v)=p}(1+v) = \prod_{p \in
    P}c_{\leq s}(f^{-1}p)$. 

  If $f^{-1}p$ contains $t>s$ elements for some $p \in P$, then the
  product of these $t$ elements with the same color $p$ is not a
  summand in $c_{\leq s}(f^{-1}p)$ but it is a summand in $c(V)$. Thus
  $c(V)$ and $\prod_{p \in P} c_{\leq s}(f^{-1}p)$ do not have the
  same homogenous components of degree $2t$.
\end{proof}

\begin{proof}[Proof of Theorem~\ref{thm:Pscol}]
  We have
  \begin{align*}
    \text{\func fVP is a $(P,s)$-coloring} &\iff
    \text{$\forall \sigma \in K \colon f \circ i_{\sigma}$ is
      $s$-to-$1$} \\ &\iff
    \forall \sigma \in K \colon c({\sigma}) = 
    \sum_{p \in P} c_{\leq s}((f \circ i_{\sigma})^{-1}p) \\&\iff
    \forall \sigma \in K \colon \SR{i_{\sigma}}\big( c(V) \big) =
    \SR{i_{\sigma}}\big(\sum_{p \in P} c_{\leq s}(f^{-1}p) \big) \\&\iff
    c(\lambda_K) = \sum_{p \in P} c_{\leq s}(f^{-1}p)
  \end{align*}
  from Lemma~\ref{lemma:sto1inRV} and \ref{lemma:charcK}.
\end{proof}

In other words, a partition $V=V_1 \cup \cdots \cup V_r$ of $V$ into
$r$ disjoint nonempty subsets is an $(r,s)$-coloring of $K$ if and
only if the equation
  \begin{equation*}
    c(V) = \prod_{1 \leq j \leq r}c_{\leq s}(V_j)
  \end{equation*}
holds in the  Stanley--Reisner algebra $\SR{K;R}$.

\begin{exmp}\label{exmp:P2}[Surfaces of genus one] 
  As indicated in Figure~\ref{fig:P2col}, there is a triangulation
  $\mathrm{P2}$ of $\R P^2$ with vertex set $V=\{v_1,\ldots,v_6\}$ and
  $f$-vector $( 1, 6, 15, 10 )$.  The chromatic numbers $\chs
  1{\mathrm{P2}} = 6$ and $\chs 2{\mathrm{P2}} = 3$. $\mathrm{P2}$ has
  $6!=720$ $(6,1)$-colorings and $45 \cdot 3! = 270$
  $(3,2)$-colorings.  Since $[1,2,3,4,5,6] \to [1,1,1,2,2,3]$ is a
  $(3,2)$-coloring, the identity
  \begin{equation*}
    \prod_{1 \leq i \leq 6}(1+v_i) =
    (1+\textcolor{red}{v_1+v_2+v_3+v_2v_3+v_1v_3+v_1v_2})
    (1+\textcolor{blue}{v_4+v_5+v_4v_5})(1+\textcolor{green}{v_6}) 
  \end{equation*}
  holds in the Stanley--Reisner ring
  \begin{equation*}
    \SR{\mathrm{P2};\Z} =
    \Z[v_1,\ldots,v_6]/(v_1v_2v_3, v_1v_2v_5, v_1v_3v_6, v_1v_4v_5,
    v_1v_4v_6, v_2v_3v_4, v_2v_4v_6, v_2v_5v_6, v_3v_4v_5, v_3v_5v_6) 
  \end{equation*}
  of $\mathrm{P2}$.

  Figure~\ref{fig:P2col} also indicates a $(3,2)$-coloring of
  $\mathrm{T2}$, a triangulation of the torus with $f$-vector $(1, 7,
  21, 14)$ (M\"obius' vertex minimal triangulation).  The chromatic
  numbers in this case are $\chs 1{\mathrm{T2}} = 7$ and $\chs
  2{\mathrm{T2}} = 3$.  $\mathrm{T2}$ has $7!=5040$ $(7,1)$-colorings
  and $84 \cdot 3! = 504$ $(3,2)$-colorings.
\end{exmp}

\begin{figure}[t]
  \centering

\begin{equation*}
 \xy 0;/r.11pc/:
 (10,0)*{\textcolor{red}{\blacksquare}}="A"+(-10,0)*{2};
 (60,-30)*{\textcolor{red}{\blacksquare}}="B"+(0,-10)*{1};
 (110,0)*{\textcolor{red}{\blacksquare}}="C"+(10,0)*{3};
 (10,40)*{\textcolor{red}{\blacksquare}}="D"+(-10,0)*{3};
 (40,40)*{\textcolor{blue}{\blacksquare}}="E"+(0,10)*{5};
 (80,40)*{\textcolor{green}{\blacksquare}}="F"+(0,10)*{6};
 (110,40)*{\textcolor{red}{\blacksquare}}="G"+(10,0)*{2};
 (60,0)*{\textcolor{blue}{\blacksquare}}="H"+(0,10)*{4};
 (60,70)*{\textcolor{red}{\blacksquare}}="I"+(0,10)*{1};
  "A"; "B" **\dir{-};
  "B"; "C" **\dir{-};
  "C"; "H" **\dir{-};
  "A"; "H" **\dir{-};
  "A"; "D" **\dir{-};
  "D"; "E" **\dir{-};
  "A"; "E" **\dir{-};
  "E"; "F" **\dir{-};
  "E"; "I" **\dir{-};
  "F"; "I" **\dir{-};
  "F"; "G" **\dir{-};
  "C"; "F" **\dir{-};
  "C"; "G" **\dir{-};
  "H"; "E" **\dir{-};
  "H"; "F" **\dir{-};
  "D"; "I" **\dir{-};
  "G"; "I" **\dir{-};
  "H"; "B" **\dir{-};
 \endxy
\qquad\qquad
\raisebox{-36pt}{
\xy 0;/r.18pc/:
  (0,60)*{\textcolor{red}{\blacksquare}}="A"+(0,7)*{1}; 
  (60,60)*{\textcolor{red}{\blacksquare}}="B"+(0,7)*{3}; 
  (20,0)*{\textcolor{blue}{\blacksquare}}="C"+(0,-7)*{4};
  (0,0)*{\textcolor{red}{\blacksquare}}="D"+(0,-7)*{1}; 
  (40,60)*{\textcolor{green}{\blacksquare}}="E"+(0,7)*{7}; 
  (60,0)*{\textcolor{red}{\blacksquare}}="F"+(0,-7)*{3};
  (20,60)*{\textcolor{blue}{\blacksquare}}="K"+(0,7)*{4}; 
  (40,0)*{\textcolor{green}{\blacksquare}}="L"+(0,-7)*{7};
"A"; "B" **\dir{-};
"D"; "F" **\dir{-};
"A"; "C" **\dir{-} \POS?(.32)*{\textcolor{red}{\blacksquare}}="I"+(-7,0)*{2}; 
"E"; "F" **\dir{-} \POS?(.68)*{\textcolor{red}{\blacksquare}}="G"+(7,0)*{2};
"D"; "E" **\dir{-} \POS?(.32)*{\textcolor{red}{\blacksquare}}="H"+(-7,0)*{3};
"B"; "C" **\dir{-} \POS?(.32)*{\textcolor{red}{\blacksquare}}="J"+(7,0)*{1};
"G"; "H" **\dir{-} \POS?(.5)*{\textcolor{blue}{\blacksquare}}="N"+(1,7)*{6};
"I"; "J" **\dir{-} \POS?(.5)*{\textcolor{blue}{\blacksquare}}="M"+(1,7)*{5};
"I"; "K" **\dir{-};
"K"; "L" **\dir{-};
"G"; "L" **\dir{-};
\endxy}
\end{equation*}
    \caption{ $(\{\textcolor{red}{\blacksquare},
      \textcolor{blue}{\blacksquare},
    \textcolor{green}{\blacksquare}\} ,2)$-colorings of $\mathrm{P2}$
    and $\mathrm{T2}$}
  \label{fig:P2col}
\end{figure}

\begin{rmk}
  The number of
  $s$-to-$1$ maps of $V=\{1,\ldots,m\}$ {\em onto\/} $P=\{1,\ldots,r\}$,
  $m \leq rs$, equals
  \begin{equation*}
     \sum 
    \begin{pmatrix}
      m \\ m_1,\ldots,m_r
    \end{pmatrix}
    \begin{pmatrix}
      r \\ r_1,\ldots,r_s
    \end{pmatrix}
  \end{equation*}
  where the sum is taken over all $r$-vectors $(m_1,\ldots,m_r)$ such
  that $s \geq m_1 \geq \cdots \geq m_r$, $m_1+\cdots+m_r=m$, and the
  $s$-vector $(r_1,\ldots,r_s)$ consists of the numbers $r_j=|\{ p \in
  P \mid m_p = j \}|$, $1 \leq j \leq s$.
\end{rmk}

\begin{exmp}
  Let $H=(V,E)$ be a $(s+1)$-uniform hypergraph with vertex set $V$
  and hyperedges $E$ \cite[Chp 1.3]{bondy_murty08}. We say that $H$ has
  property $B$ if there exists a red--blue coloring of the vertices
  with no monochrome hyperedges \cite{miller37}. Form the pure ASC
  $K(H)$ with vertex set $V$ and facets $E$. Then $K(H)$ is
  $(2,s)$-colorable if and only if $H$ has property $B$.
\end{exmp}

\subsection{Coloring flags}
\label{sec:flags}
We introduce the notion of an $s$-flagification procedure that relates
to $(P,s)$-colorings in the same way that flagification
(Definition~\ref{defn:1flag}) relates to $P$-colorings
(Proposition~\ref{prop:1flag}, Proposition~\ref{prop:1col}).

\begin{defn}\label{defn:sflag}
  The $s$-flagification of $K$ is the ASC
  \begin{equation*}
    \fla sK = \{ \sigma \in D[V] \mid \sk s{[D[\sigma]} \subset K \}
  \end{equation*}
  and $K$ is an $s$-flag complex if $K= \fla sK$.
\end{defn}

An ASC is a $1$-flag complex if and only it is a flag complex in the
sense of Definition~\ref{defn:1flag} and $\fla{}K = \fla 1K$.  The
$s$-flagification of $K$ is the largest complex on $V$ with the same
$s$-skeleton as $K$.

\begin{prop}
  The following conditions are equivalent:
  \begin{enumerate}
  \item $K$ is $s$-flag
  \item $\forall \sigma \in D[V] \colon \sk s{D[\sigma]} \subset K
    \Longrightarrow \sigma \in K$
  \item The missing faces of $K$ have dimension at most $s$
  \end{enumerate}
\end{prop}
\begin{proof}
  If  $K$ has a missing face of dimension $s+1$, then that face must
  be added to form the $s$-flagification of $K$, so $K$ is not $s$-flag.
\end{proof}

For example, the triangulation $\mathrm{P2}$ of the real projective
plane from Example~\ref{exmp:P2} is $2$-flag, and, more generally, the
$j$-skeleton $\sk j{D[V]}$ of a full simplex is $(j+1)$-flag but not
$j$-flag.

Complexes of dimension less than $s$ are $s$-flag (but the converse
does not hold).  Any $s$-flag complex is an $(s+1)$-flag complex so
that
\begin{equation*}
  \fla 1{\text{ASC}} \subset \fla 2{\text{ASC}} \subset \cdots \subset
  \fla s{\text{ASC}} \subset \fla {s+1}{\text{ASC}} \subset \cdots  
\end{equation*}
where $\fla s{\text{ASC}}$ stands for the class of $s$-flag ASCs.  The
complex $\partial D[N_+]$ is $(N-1)$-dimensional so it is $N$-flag but
it is not $(N-1)$-flag.  The complexes $\mathrm{MB}$
(Figure~\ref{fig:MBcol}) and $\mathrm{P2}$ (Figure~\ref{fig:P2col})
are $2$-flag (but not $1$-flag).

The next proposition says that $(P,s)$-coloring issues are $s$-dimensional.

\begin{prop}\label{prop:flags}
  The following conditions are equivalent for any map \func fVP:
  \begin{enumerate}
  \item $f$ is a $(P,s)$-coloring of $K$
  \item $f$ is a $(P,s)$-coloring of $\sk sK$
  \item $f$ is a $(P,s)$-coloring of $\fla sK$
  \end{enumerate}
  Moreover, $\mathrm{Col}^s(K,P) = \mathrm{Col}^s(\sk sK,P) =
  \mathrm{Col}^s(\fla sK,P)$ and $\chs sK = \chs s{\sk sK} =
  \chs s{\fla sK}$.  
\end{prop}



\subsection{Proof of the main theorem}
\label{sec:VBoverDJ}
We are now ready to prove our main result.

\begin{proof}[Proof of Theorem~\ref{thm:main}]
  
  \noindent $\text{\eqref{thm:main1}} \Longrightarrow
  \text{\eqref{thm:main3}}$: By Theorem~\ref{thm:fsplit} and vector
  bundle theory \cite[Theorem 9.6]{husemoller94}, if $K$ admits a
  $(P,s)$-coloring then $\lambda_K$ is stably isomorphic to a sum of
  complex vector bundles of dimension at most $s$.

   \noindent $\text{\eqref{thm:main3}} \Longrightarrow
  \text{\eqref{thm:main5}}$: This is clear.

  \noindent $\text{\eqref{thm:main5}} \Longrightarrow
  \text{\eqref{thm:main1}}$: Since $R$ is a UFD, also the polynomial
  rings over $R$ are UFDs by Gauss' theorem.  Suppose that $\prod_{v
    \in V} (1+v) = \prod_{p \in P}c_p$ with $c_p \in \SR{K;R}$, $p \in
  P$, of degree at most $2s$. Recall from Section~\ref{sec:SR} that
  for any vertex vertex $v \in V$,
  \func{\SR{i_{\{v\}}}}{\SR{K;R}}{\SR{D[\{v\}];R} = R[v]} denotes the
  ring homo\m\ induced by the functor $\SR{-;R}$ applied to the
  inclusion \func{i_{\{v\}}}{D[\{v\}]}K of that vertex into $K$. Since
  the equation $1+v = \prod_{q \in P} \SR{i_{\{v\}}}c_q$ holds in the
  UFD $R[v]$ there is a unique $p \in P$ such that $1+v$ divides
  $\SR{i_{\{v\}}}c_p$.  Define \func fVP by
  \begin{equation*}
   \forall v \in V \forall p \in P \colon f(v)=p \iff 1+v \mid
   \SR{i_{\{v\}}}c_p  
  \end{equation*}
  Suppose that $\sigma$ is a simplex of $K$.
  The equation
  \begin{equation*}
    \prod_{v \in \sigma} 1+v = \prod_{p \in P} \SR{i_\sigma}c_p
  \end{equation*}
 holds in the UFD $R[\sigma]$. For each $v \in \sigma$, the prime
 element $1+v$ divides exactly one of the factors
 $\SR{i_\sigma}c_p$, $p \in P$. The only possibility is that $1+v$
 divides $\SR{i_\sigma}c_{f(v)}$. It follows that
 \begin{equation*}
    \forall p \in P \colon 
    \prod_{v \in \sigma \cap f^{-1}p} (1+v) \mid \SR{i_\sigma}c_p
 \end{equation*}
 and for degree reasons we must in fact have that
 \begin{equation*}
    \forall p \in P \colon 
    \prod_{v \in \sigma \cap f^{-1}p} (1+v) = \SR{i_\sigma}c_p
 \end{equation*}
 up to a unit in $R$.  Therefore $2|\sigma \cap f^{-1}p| =
 \deg(\SR{i_\sigma}c_p) \leq \deg(c_p) \leq 2s$ or $|\sigma \cap
 f^{-1}p| \leq s$. This means that \func fVP is a $(P,s)$-coloring of
 $K$.
%
%
%

 \noindent $\text{\eqref{thm:main1}} \iff \text{\eqref{thm:main4}}$:
 This is the special case of Theorem~\ref{thm:gen} with $L=D[P]$. Note
 that the proof of Theorem~\ref{thm:gen} relies only on items
 \eqref{thm:main1}--\ref{thm:main3} from Theorem~\ref{thm:main}.
\end{proof}


There is a version of Theorem~\ref{thm:main} that refers to $\xi_K$
rather than $\lambda_K$. Since only complexes $K$ with $\dim K < rs$
admit $(r,s)$-colorings \eqref{eq:chrs} it is no restriction to make
this assumption.  

\begin{cor}\label{cor:lKxK}
  Assume that $n(K) \leq rs$. Then $K$ admits an $(r,s)$-coloring if
  and only if there exists a map $\DaJ{K} \to \BU s^r$ such that the
  diagram
\begin{equation*}
  \xymatrix{
    {\DaJ{K}} \ar[dr]_{\xi_K \oplus \C^{rs-n(K)}} \ar[rr] &&  
    {\BU s^r} \ar[dl]^\oplus \\
    & {\BU{rs}}}
\end{equation*}
  is homotopy commutative.
\end{cor}
\begin{proof}
  If the diagram has a completion, then $K$ admits an $(r,s)$-coloring
  by Theorem~\ref{thm:main}. Conversely, if $K$ admits an
  $(r,s)$-coloring, by the short exact sequence \eqref{eq:xisumpxi}
  there are $r$ vector bundles $\xi_p$, $1 \leq p \leq r$, over $\DaJ
  K$ such that
  \begin{equation*}
    \bigoplus_{1 \leq p \leq r} \xi_p \cong \xi_K \oplus \C^{rs-n(K)}
  \end{equation*}
  where $\dim \xi_p = s$, $1 \leq p \leq r$. 
\end{proof}

\subsection{$(L,s)$-colorings of simplicial complexes}
\label{sec:general}

Let $K$ be an ASC with vertex set $V$ and $L$ an ASC with vertex set
$P$.

\begin{defn}
  An $(L,s)$-coloring of $K$ is a simplicial map \func fKL whose
  vertex map \func fVP is a $(P,s)$-coloring of $K$.
\end{defn}


A $(P,s)$-coloring is the same thing as a $(D[P],s)$-coloring.  Any
$(L,s)$-coloring is a $(P,s)$-coloring.  A $(P,s)$-coloring of $K$ is
a $(\partial D[P],s)$-coloring of $K$ if no simplex of $K$ uses the
full palette $P$. There are no
$(\partial[\{\textcolor{blue}{\blacksquare},
\textcolor{green}{\blacksquare}, \textcolor{red}{\blacksquare}\}]
,2)$-colorings of $\mathrm{P2}$ from Figure~\ref{fig:P2col} which
means that it is impossible to paint the vertices of $\mathrm{P2}$
from a palette of $3$ colors so that every facet has exactly $2$
colors.

Let
\begin{equation*}
  \mathrm{map}^s(V,P) = \{ \func fVP \mid \forall p \in P \colon
  |f^{-1}p| \leq s \}, \qquad
  \mathrm{map}^s(V,L) = \{ f \in \mathrm{map}^s(V,P) \mid f(V) \in L \}
\end{equation*}
be the set of at most $s$-to-$1$ maps from $V$ to $P$ and the subset
of those at most $s$-to-$1$ maps from $V$ to $P$ whose image is a
simplex of $L$. Then $\mathrm{map}^s(V,L)$ is the set of
$(L,s)$-colorings of the full simplex $D[V]$. More generally, let
$\mathrm{Col}^s(K,L)$ stand for the set of $(L,s)$-colorings of the
ASC $K$. Then
\begin{equation*}
  \mathrm{Col}^s(K,L) = 
  {\lim}(P(K)^{\mathrm{op}};\mathrm{map}^s(-,L)) =
  {\lim}(P(K)^{\mathrm{op}};\mathrm{Col}^s(D[-],L))
\end{equation*}
is the limit of the contravariant functor from $P(K)$ to sets that
takes any simplex $\sigma$ of $K$ to $\mathrm{map}^s(\sigma,L) =
\mathrm{Col}^s(D[\sigma],L)$.

For each $p \in P$ let $\nu_p$ denote the $s$-dimensional complex
vector bundle
\begin{equation}\label{eq:lambdap}
  \Map{}P{P-\{p\}}{\VU s}0 \to   \map{}P{\BU s}
\end{equation} 
classified by the evaluation map \func{\nu_p}{\map{}P{\BU s}}{\BU
  s} at $p$. Let $\nu_P = \bigoplus_{p \in P} \nu_p$ be the
$|P|s$-dimensional Whitney sum
\begin{equation*}
  \map{}P{\VU s} \to \map{}P{\BU s},
\end{equation*}
classified by the map $ \nu_P \colon \map{}P{\BU s} \xrightarrow{\prod
  \nu_p} \prod_{p \in P} \BU s \xrightarrow{\oplus} \BU{|P|s}$, of
these bundles.

\begin{lemma}\label{lemma:localLscol}
  There exists a map $f \in \mathrm{map}^s(V,L) =
  \mathrm{Col}^s(D[V],L)$ if and only if there exists a map
  \begin{equation*}
    \xymatrix{
      {\map{}V{\BU 1}} \ar[dr]_-{\oplus} \ar[rr]^-F &&
      {\DaJ{L;\BU s,*}} \ar[dl]^-{\oplus \circ \lambda_L} \\
      & {\mathrm{BU}} }
  \end{equation*}
  over $\mathrm{BU}$. If this is the case, then $F^*c(\nu_p) =
  \pm c(f^{-1}p)$ in $H^*(\map{}V{\BU 1};\Z) = \Z[V]$ for all $p \in P$.
\end{lemma}
\begin{proof}
  Suppose that $f \in \mathrm{map}^s(V,L)$ is an $s$-to-$1$-map.  Pick
  any linear ordering on $V$ and use it to define Lie group homo\m s
  $\map{}{f^{-1}p}{\U 1} \to \U{|f^{-1}p|} \to \U s$. Apply the
  classifying space functor to get maps $\mu_p \colon
  \map{}{f^{-1}p}{\BU 1} \to \BU{|f^{-1}p|} \to \BU s$, $p \in P$. Set
  $\mu = \{\mu_p \mid p \in P\}$. The induced map \func{\DaJ{f;\mu}}
  {\map{}V{\BU 1}=\DaJ{D[V];\BU 1,*}} {\DaJ{L;\BU s,*}} from
  Proposition~\ref{prop:DJfunctor}.\eqref{prop:DJfunctor4} will make
  the diagram homotopy commutative.

  Conversely, given a map $F \colon \map{}V{\BU 1} \to \DaJ{L;\BU
    s,*}$ over $\mathrm{BU}$, we get a map $\map{}V{\BU 1} \to
  \map{}P{\BU s}$ over $\mathrm{BU}$ by composing with the inclusion
  map \func{\lambda_L}{\DaJ{L;\BU s,*}}{\map{}P{\BU s}} introduced
  immediately below Definition~\ref{defn:DJ}. Then
  \begin{equation*}
    \prod_{v \in V}(1+v) = \prod_{p \in P}c_p
  \end{equation*}
  with $c_p=F^*\lambda_L^*c(\nu_p)$, $p \in P$.
  According to (the proof of)
  Theorem~\ref{thm:main}.\eqref{thm:main1}--\eqref{thm:main3}, there
  is a map $f \in \mathrm{map}^s(V,P)$ so that
  \begin{equation*}
    \forall p \in P \colon \prod_{v \in f^{-1}p}(1+v) = \pm c_p 
  \end{equation*}
  in the polynomial ring $\Z[V] = H^*(\map{}V{\BU 1};\Z)$.  We claim
  that $f(V) = \{ p \in P \mid c_p \neq 1 \}$ is a simplex of
  $L$, or, equivalently, that $f \in \mathrm{map}^s(V,L)$.
  
  Assume that $f(V) \not\in L$.  According to \cite[Theorem
  3.10]{NR05} (see Proposition~\ref{prop:HDJR})
  \begin{equation*}
    H^*(\DaJ{L;\BU s,\ast};\Z) =
    \lim(P(L)^\mathrm{op}; \otimes\Map{}P{P-?}{H^*(\BU s;\Z)}{\Z})
  \end{equation*}
  The element $\prod_{p \in f(V)} (c(\nu_p)-1)$ is in the
  kernel of 
  \begin{equation*}
    H^*(\map{}P{\BU s};\Z) \xrightarrow{\lambda_L^*} 
    H^*(\DaJ{L;\BU s,\ast};\Z) \subset
    \prod_{\tau \in L}\otimes\Map{}P{P-\tau}{H^*(\BU
    s;\Z)}{\Z})
  \end{equation*}
  because $\tau - f(V) \neq \emptyset$ for any simplex $\tau$ of
  $L$. Thus
  \begin{equation*}
    \prod_{p \in f(V)} \lambda_L^*c(\nu_p) = 
    \sum_{p \in Q \subsetneq f(V)} (-1)^{|f(V)|-|Q|}\lambda_L^*c(\nu_p)
  \end{equation*}
  in $H^*(\DaJ{L;\BU s,\ast};\Z)$ and, in $H^*(\map{}V{\BU 1};\Z) = \Z[V]$,
  \begin{multline*}
    \prod_{v \in V}(1+v) = 
    \prod_{p \in f(V)} c_p =
    F^*\big(\prod_{p \in f(V)} \lambda_L^*c(\nu_p) \big)  = 
    F^*\big(\sum_{p \in Q \subsetneq f(V)}
    (-1)^{|f(V)|-|Q|}\lambda_L^*c(\nu_p) \big) \\ = 
    \sum_{p \in Q \subsetneq f(V)} (-1)^{|f(V)|-|Q|}c_p     =
    \sum_{Q \subsetneq f(V)} \pm (-1)^{|f(V)|-|Q|}  
    \prod_{v \in f^{-1}Q}(1+v)
  \end{multline*}
  which is a contradiction because these elements of $\Z[V]$ do not
  have the same homogeneous components in degree $2|V|$. Therefore we
  must have that $f(V) \in L$.
\end{proof}

\begin{thm}\label{thm:gen}
  The ASC $K$ is $(L,s)$-colorable if and only if there exists a map
  $\DaJ{K;\BU 1,*} \to \DaJ{L;\BU s,*}$ such that the diagram
  \begin{equation*}
    \xymatrix{
      {\DaJ{K;\BU 1,*}} \ar[dr]_{\oplus \circ \lambda_K} \ar[rr] &&
      {\DaJ{L;\BU s,*}} \ar[dl]^{\oplus \circ \lambda_L} \\
      & {\mathrm{BU}} }
  \end{equation*}
  commutes up to homotopy.
\end{thm}
\begin{proof}
  Suppose that \func fVP is an $(L,s)$-coloring of $K$. Pick any
  linear ordering on $V$ and use it to define Lie group homo\m s
  $\map{}{f^{-1}p}{\U 1} \to \U{|f^{-1}p|} \to \U s$. Apply the
  classifying space functor to get maps $\mu_p \colon
  \map{}{f^{-1}p}{\BU 1} \to \BU{|f^{-1}p|} \to \BU s$, $p \in P$. Set
  $\mu = \{\mu_p \mid p \in P\}$. The induced map
  \func{\DaJ{f;\mu}}{\DaJ{K;\BU 1,*}}{\DaJ{L;\BU s,*}} from
  Proposition~\ref{prop:DJfunctor}.\eqref{prop:DJfunctor4} will make
  the diagram homotopy commutative.

  Conversely,  given a map $\DaJ{K;\BU 1,*} \to
  \DaJ{L;\BU s,*}$ over $\mathrm{BU}$. Let $\sigma$ be a simplex of $K$
  and $f_{\sigma} \in \mathrm{map}^s(\sigma,L) =
  \mathrm{Col}^s(D[\sigma],L)$ the local coloring map that according
  to Lemma~\ref{lemma:localLscol} corresponds to the map
  \begin{equation*}
    \map{}{\sigma}{\BU 1} =
    \DaJ{D[\sigma];\BU 1,*} \xrightarrow{\DaJ{i_{\sigma}}}
    \DaJ{K;\BU 1,*} \to
    \DaJ{L;\BU s,*}
  \end{equation*}
  over $\mathrm{BU}$. Then $f=(f_{\sigma})_{\sigma \in K} \in
 \mathrm{Col}^s(K,L) =
 {\lim}(P(K)^{\mathrm{op}};\mathrm{Col}^s(D[\sigma],L))$ is an
 $(L,s)$-coloring of $K$.
\end{proof}

\section{Vertex colorings of polyhedra}
\label{sec:speculations}


Let $M$ be a simplicial manifold, or, more generally, a compact
polyhedron; see \cite[Definitions 2.31 and 2.33]{BP02} for definitions
of the terminology applied here.

\begin{defn}\label{defn:chrSd}
  $\chs s{M} = \max\{ \chs sK \mid \text{$K$ triangulates $M$}\}$ is
  the maximum of $\chs sK$ \eqref{defn:scol} over all finite
  triangulations $K$ of the polyhedron $M$.
\end{defn}

The chromatic number is an invariant of the the homeo\m\ type, but
not an invariant of the homotopy type as, for instance, triangulations
of the plane are very different from triangulations of a point.  The
assertion $\chs s{S^d} = r$ means that one needs at most $r$ colors to
color {\em any\/} simplicial  $d$-sphere in such a way that the $s$-dimensional
simplices are polychrome. For instance,
\begin{alignat*}{4}
  &\chs 2{S^1} = 1 \qquad & &\chs 1{S^1} = 3 \\
  &\chs 3{S^2} = 1 \qquad & &\chs 2{S^2} = 2 \qquad & &\chs 1{S^2} = 4 \\
  &\chs 4{S^3} = 1 \qquad & &\chs 3{S^3} \geq 2 \qquad & &\chs 2{S^3}
  \geq 4 \qquad & &\chs 1{S^3} = \infty
\end{alignat*}
Example~\ref{exmp:41coldisc} implies that $\chs 1{\Sigma} \geq 4$ for
all compact surfaces $\Sigma$ (and the exact values are, except for
the $2$-sphere and the Klein bottle, given by Heawood's inequality and
the Map color theorem \cite{ringel_youngs68}).  The $4$-color theorem
\cite[Chp 11]{bondy_murty08} says that $\chs 1{S^2}=4$. From a
$(4,1)$-coloring of $S^2$ we get a $(2,2)$-coloring by identifying
colors pairwise. Thus $\chs 2{S^2} \leq 2$ and the reverse inequality
is \eqref{eq:chrs}. The assertion about the chromatic numbers of $S^3$
are justified by the examples below.

\begin{exmp}
  The simplicial $3$-sphere $\partial D[3_+] \ast \partial D[1_+]$ is
  $(3,2)$-colorable because it has $6$ vertices \eqref{eq:chrs}.
  $\partial D[3_+] \ast \partial D[1_+]$ is a $3$-flag complex.
\end{exmp}

\begin{exmp}\label{exmp:altschuler}
  A computer verification shows that Altschuler's \lq peculiar\rq\
  simplicial $3$-sphere $\mathrm{ALT}$ \cite{ALT76}, with $f$-vector
  $( 1, 10, 45, 70, 35 )$, is a $2$-flag $(4,2)$-colorable ASC with
  coloring $[1, 1, 2, 2, 1, 2, 3, 3, 4, 4]$. The $2$-chromatic number
  is $\chs 2{\mathrm{ALT}} = 4$.
\end{exmp}

\begin{exmp}\label{exmp:MW}
  A computer verification shows that Mani and Walkup's simplicial
  $3$-spheres $C$ and $D$ \cite{MW80}, with $f$-vectors
  $(1,20,126,212,106)$ and $(1,16,106,180,90)$, are $2$-flag complexes
  with $2$-chromatic numbers $\chs 2C=3$ and $\chs 2D=4$.
\end{exmp}

\begin{exmp}
  Klee and Kleinschmidt constructed a simplicial $3$-sphere
  $\mathrm{KK}$ with $f$-vector $(1,16,106,180,90)$ that is not
  vertex-decomposable and not polytopal \cite[Table 1]{kleeklein87}.
  This $3$-spherical complex is $2$-flag, $\chs 2 {\mathrm{KK}} = 4$
  and $[ 1, 3, 1, 3, 2, 2, 2, 2, 4, 4, 1, 1, 3, 2, 2, 4 ]$ is a
  $(4,2)$-coloring.
\end{exmp}

These examples show that $\chs 2{S^3} \geq 4$. All these simplicial
$3$-spheres are $(2,3)$-colorable so that $\chs 3{S^3} \geq 2$.  We do
not know if $\chs 2{S^3}$ or $\chs 3{S^3}$ are finite numbers.

\begin{exmp}\label{exmp:lutz}
  The non-constructible simplicial $3$-spheres $S^3_{17,74}$ and
  $S^3_{13,65}$ with $f$-vectors $(1, 17, 91, 148, 74 )$ and $(1, 13,
  69, 112, 56)$ found by Lutz \cite{lutz04} are
  $(3,2)$-colorable and $\chs 2{S^3_{17,74}} = 3 = \chs
  2{S^3_{13,65}}$.  The non-constructible simplicial $3$-sphere with
  $f$-vector $f=(1,381,2309,3856,1928)$ containing a knotted triangle
  \cite{lickorish91,hachimori_ziegler2000} is $(4,2)$-colorable.
  The reduction modulo $4$ map $\{1,2,\ldots,4m\} \to \{0,1,2,3\}$ is
  a $(4,2)$-coloring of the centrally symmetric simplicial $3$-spheres
  $CS^3_{4m}$, $m \geq 2$, from \cite[Theorem 6]{lutz-2004}.
  The triangulations $^3_{nn}10^{di}_{1}$ and $^3_{nn}12^{cy}_{1}$
  from \cite[Table 2]{lutz-2004} are $(3,2)$-colorable,
  $^3_{nn}14^{cy}_{1}$ is $(4,2)$-colorable, and
  ${}^3_{nn}16^{cy}_{k}$, $1 \leq k \leq 5$, are $(4,2)$-colorable by
  the map $v \to v \bmod 4$, $1 \leq v \leq 16$.

  $[1, 1, 1, 2, 2, 3, 3, 1, 3, 3, 4, 2, 2, 4, 4, 1]$ is a
  $(4,2)$-coloring of the $2$-flag triangulation, $\mathrm{BLP}$, of
  the Poincare homology $3$-sphere with $f=( 1, 16, 106, 180, 90 )$
  constructed by Bj\"orner and Lutz \cite{bjorner-lutz00} and $\chs
  2{\mathrm{BLP}}=4$.
\end{exmp}

\begin{exmp}
  $\partial D[(2n)_+]$ is a triangulation of $S^{2n-1}$, $n
  \geq 1$,
  with $n$-chromatic number $\chs n{\partial D[(2n)_+]} = 3$. The map
  \func{f}{\{0,1,\ldots,2n\}}{\{1,2,3\}} given by
  \begin{equation*}
    f(j) =
    \begin{cases}
      1 & 0 \leq j < n \\
      2 & n \leq j < 2n \\
      3 & j = 2n
    \end{cases}
  \end{equation*}
  is a $(3,n)$-coloring.  Similarly, the ASC $\partial D[(2n+1)_+]$ is
  a triangulation of $S^{2n}$, $n \geq 1$, with $n$-chromatic number
  \begin{equation*}
    \chs n{\partial D[(2n+1)_+]}  =
    \begin{cases}
      4 & n= 1 \\ 3 & n>1
    \end{cases}
  \end{equation*}
  and the map \func{f}{\{0,1,\ldots,2n+1\}}{\{1,2,3\}} given by
  \begin{equation*}
    f(j) =
    \begin{cases}
      1 & 0 \leq j < n \\
      2 & n \leq j < 2n \\
      3 & j = 2n,2n+1
    \end{cases}
  \end{equation*}
  is a $(3,n)$-coloring for $n>1$.
\end{exmp}

We now define the spherical complexes associated to cyclic
$n$-polytopes.

\begin{defn}
  $\mathrm{CP}(m,n)$, $m>n$, is the $(n-1)$-dimensional ASC on the
  ordered set $V=\{1,\ldots,m\}$ with the following facets: An
  $n$-subset $\sigma$ of $V$ is a facet if and only if between any two
  elements of $V-\sigma$ there is an even number of vertices
  in $\sigma$.
\end{defn}

By Gale's Evenness Theorem \cite{gale63}, the ASC $\mathrm{CP}(m,n)$
triangulates the boundary of the cyclic $n$-polytope on $m$ vertices.
Thus $\mathrm{CP}(m,n)$ is a simplicial $(n-1)$-sphere on $m$ vertices
and it is $\lfloor n/2 \rfloor$-neighborly in the sense that
$\mathrm{CP}(m,n)$ has the same $s$-skeleton as the full simplex on
its vertex set when $s < \lfloor n/2 \rfloor$
(Proposition~\ref{prop:flags}) \cite[Definition~1.15, Example~1.17,
Remark after Theorem~6.33]{BP02}. When $s < \lfloor n/2 \rfloor$, the
$s$-chromatic number of the cyclic polytope $\mathrm{CP}(m,n)$, $\chs
s{\mathrm{CP}(m,n)} = \chs s{D[V]} = \lceil m/s \rceil$, grows with
the number of vertices $m$. The first interesting chromatic numbers
for the cyclic polytopes are $\chs {n}{\mathrm{CP}(m,2n)}$ and $\chs
{n}{\mathrm{CP}(m,2n+1)}$ and the first first interesting chromatic
numbers for the spheres \eqref{defn:chrSd} are $\chs n{S^{2n-1}}$ and
$\chs n{S^{2n}}$.



\begin{prop} For $n \geq 1$,
  \begin{equation*}
      \chs n{\mathrm{CP}(m,2n)} =
    \begin{cases}
      2 & \text{$m>2n$ even} \\
      3 & \text{$m>2n$  odd}
    \end{cases} \qquad \text{and} \qquad
     \chs n{\mathrm{CP}(m,2n+1)} =
     \begin{cases}
       4 & n=1, \quad m>3 \\
       3 & n>1, \quad m> 2n+1
     \end{cases}
\end{equation*}
\end{prop}
\begin{proof}
  For general reasons \eqref{eq:chrs}, $\chs n{\mathrm{CP}(m,2n)} \geq
  2$ and $\chs n{\mathrm{CP}(m,2n+1)} \geq 3$. 

  The reduction modulo $2$ map $ \{1,\ldots, m\} \to \{0,1\}$ is a
  $(2,n)$ coloring of $\mathrm{CP}(m,2n)$, $n \geq 1$, for even $m>2n$
  and also for odd $m>2n$ if we modify this map so that it
  assume the value $2$ at vertex $m$.

  The map $f \colon \{1,\ldots, m\} \to \{0,1,2,3\}$ given by
  \begin{equation*}
    f(j) =
    \begin{cases}
      2 & j=1 \\
      j \bmod 2 & 1<j<m \\
      3 & j=m
    \end{cases}
  \end{equation*}
  is a $(4,1)$-coloring of the simplicial $2$-sphere
  $\mathrm{CP}(m,3)$, $m>3$. There is no $(4,1)$-coloring because
  there is no  $(4,1)$-coloring of $\mathrm{CP}(4,3)=\partial D[3_+]$. 

  The maps $f \colon \{1,\ldots, m\} \to \{0,1,2\}$ given by
  \begin{equation*}
    f(j) =
    \begin{cases}
      2 & j = 1,2,m \\
      j \bmod 2 & n < j < m
    \end{cases} \qquad 
    f(j) =
    \begin{cases}
      2 & j = 1,m \\
      j \bmod 2 & n < j < m
    \end{cases}
  \end{equation*}
  when $n$ is odd or $n$ is even, respectively, are $(3,n)$-colorings
  of $\mathrm{CP}(m,2n+1)$, $n >1$, $m> 2n+1$.
\end{proof}

There is no example contradicting the statement that $\chs n{S^{2n}} =
4$ for all $n \geq 1$ and $\chs n{S^{2n-1}} = 4$ for all $n \geq 2$.

\providecommand{\bysame}{\leavevmode\hbox to3em{\hrulefill}\thinspace}
\providecommand{\MR}{\relax\ifhmode\unskip\space\fi MR }
\providecommand{\MRhref}[2]{%
  \href{http://www.ams.org/mathscinet-getitem?mr=#1}{#2}
}
\providecommand{\href}[2]{#2}


\begin{thebibliography}{10}

\bibitem{ALT76}
A.~Altshuler, \emph{A peculiar triangulation of the {$3$}-sphere}, Proc. Amer.
  Math. Soc. \textbf{54} (1976), 449--452. \MR{MR0397744 (53 \#1602)}

\bibitem{bjorner-lutz00}
Anders Bj{\"o}rner and Frank~H. Lutz, \emph{Simplicial manifolds, bistellar
  flips and a 16-vertex triangulation of the {P}oincar\'e homology 3-sphere},
  Experiment. Math. \textbf{9} (2000), no.~2, 275--289. \MR{MR1780212
  (2001h:57026)}

\bibitem{bondy_murty08}
J.~A. Bondy and U.~S.~R. Murty, \emph{Graph theory}, Graduate Texts in
  Mathematics, vol. 244, Springer, New York, 2008. \MR{MR2368647 (2009c:05001)}

\bibitem{bourbaki98}
Nicolas Bourbaki, \emph{Algebra {I}. {C}hapters 1--3}, Elements of Mathematics
  (Berlin), Springer-Verlag, Berlin, 1998, Translated from the French, Reprint
  of the 1989 English translation [ MR0979982 (90d:00002)]. \MR{MR1727844}

\bibitem{BP02}
Victor~M. Buchstaber and Taras~E. Panov, \emph{Torus actions and their
  applications in topology and combinatorics}, University Lecture Series,
  vol.~24, American Mathematical Society, Providence, RI, 2002. \MR{MR1897064
  (2003e:57039)}

\bibitem{gale63}
David Gale, \emph{Neighborly and cyclic polytopes}, Proc. Sympos. Pure Math.,
  Vol. VII, Amer. Math. Soc., Providence, R.I., 1963, pp.~225--232.
  \MR{MR0152944 (27 \#2915)}

\bibitem{hachimori_ziegler2000}
Masahiro Hachimori and G{\"u}nter~M. Ziegler, \emph{Decompositons of simplicial
  balls and spheres with knots consisting of few edges}, Math. Z. \textbf{235}
  (2000), no.~1, 159--171. \MR{MR1785077 (2001m:52017)}

\bibitem{husemoller94}
Dale Husemoller, \emph{Fibre bundles}, third ed., Graduate Texts in
  Mathematics, vol.~20, Springer-Verlag, New York, 1994. \MR{MR1249482
  (94k:55001)}

\bibitem{kleeklein87}
Victor Klee and Peter Kleinschmidt, \emph{The {$d$}-step conjecture and its
  relatives}, Math. Oper. Res. \textbf{12} (1987), no.~4, 718--755.
  \MR{MR913867 (89a:52018)}

\bibitem{lickorish91}
W.~B.~R. Lickorish, \emph{Unshellable triangulations of spheres}, European J.
  Combin. \textbf{12} (1991), no.~6, 527--530. \MR{MR1136394 (92k:57044)}

\bibitem{lutz04}
Frank~H. Lutz, \emph{Small examples of nonconstructible simplicial balls and
  spheres}, SIAM J. Discrete Math. \textbf{18} (2004), no.~1, 103--109
  (electronic). \MR{MR2112491 (2005i:57028)}

\bibitem{lutz-2004}
Frank~H. Lutz, \emph{Triangulated manifolds with few vertices: Centrally
  symmetric spheres and products of spheres}, 2004.

\bibitem{MW80}
Peter Mani and David~W. Walkup, \emph{A {$3$}-sphere counterexample to the
  {$W\sb{v}$}-path conjecture}, Math. Oper. Res. \textbf{5} (1980), no.~4,
  595--598. \MR{MR593649 (82a:52004)}

\bibitem{miller37}
E.W. Miller, \emph{On a property of families of sets}, C.R. Soc. Sci. Varsovie
  (1937), 31--38.

\bibitem{DN08}
Dietrich Notbohm, \emph{Colorings of simplicial complexes and vector bundles
  over {D}avis--{J}anuszkiewicz spaces},
  \href{http://xxx.lanl.gov/abs/0905.4454v1}{arXiv:0905.4454v1}, 2008.

\bibitem{DN09}
\bysame, \emph{Vector bundles over {D}avis--{J}anuszkiewicz spaces with
  prescribed characteristic classes},
  \href{http://xxx.lanl.gov/abs/0905.4461}{arXiv:0905.4461v1}, 2009.

\bibitem{NR05}
Dietrich Notbohm and Nigel Ray, \emph{On {D}avis-{J}anuszkiewicz homotopy
  types. {I}. {F}ormality and rationalisation}, Algebr. Geom. Topol. \textbf{5}
  (2005), 31--51 (electronic). \MR{MR2135544 (2006a:55016)}

\bibitem{PRV04}
Taras Panov, Nigel Ray, and Rainer Vogt, \emph{Colimits, {S}tanley-{R}eisner
  algebras, and loop spaces}, Categorical decomposition techniques in algebraic
  topology (Isle of Skye, 2001), Progr. Math., vol. 215, Birkh\"auser, Basel,
  2004, pp.~261--291. \MR{MR2039770 (2004k:55008)}

\bibitem{ringel_youngs68}
Gerhard Ringel and J.~W.~T. Youngs, \emph{Solution of the {H}eawood
  map-coloring problem}, Proc. Nat. Acad. Sci. U.S.A. \textbf{60} (1968),
  438--445. \MR{MR0228378 (37 \#3959)}

\end{thebibliography}
\end{document}